\documentclass[review,3p,square]{elsarticle}
\usepackage{amssymb}
\usepackage{amsmath}
\usepackage{amsthm}
\usepackage{stmaryrd}
\usepackage{natbib}
\usepackage{bbm}
\usepackage{amsmath}
\allowdisplaybreaks[4]
\usepackage{enumitem}
\usepackage{fancyhdr}
\usepackage{hyperref}
\hypersetup{%
  pdftitle={BSDEs},%
  pdfsubject={BSDEs},%
  pdfauthor={XinYing Li, ShengJun FAN},%
  pdfkeywords={BSDEs},%
  pdfstartview=FitH,%
  CJKbookmarks=true,%
  bookmarksnumbered=true,%
  bookmarksopen=true,%
  colorlinks=true, linkcolor=blue, urlcolor=blue, citecolor=blue, %
}

\usepackage{cleveref}
\crefname{thm}{Theorem}{Theorems}
\crefname{pro}{Proposition}{Propositions}
\crefname{lem}{Lemma}{Lemmas}
\crefname{rmk}{Remark}{Remarks}
\crefname{cor}{Corollary}{Corollaries}
\crefname{dfn}{Definition}{Definitions}
\crefname{ex}{Example}{Examples}
\crefname{section}{Section}{Sections}
\crefname{subsection}{Subsection}{Subsections}

\newcommand{\F}{\mathcal{F}}
\newcommand{\E}{\mathbb{E}}

\newcommand{\essinf}{\mathop{\operatorname{ess\,inf}}}
\newcommand{\R}{{\mathbb R}}

\newcommand {\Dis}{\displaystyle}

\newtheorem{thm}{Theorem}[section]

\newtheorem{pro}[thm]{Proposition}
\newtheorem{rmk}[thm]{Remark}
\newtheorem{cor}[thm]{Corollary}
\newtheorem{dfn}[thm]{Definition}
\newtheorem{ex}[thm]{Example}

\journal{arXiv}
\begin{document}
\begin{frontmatter}

\title{{Weighted $L^p~(p\geq1)$ solutions of random time horizon BSDEs with stochastic monotonicity generators}\tnoteref{found}}
\tnotetext[found]{Supported by National Natural Science Foundation of China (Nos. 12171471 and 12631016).
\vspace{0.2cm}}

\author{{Xinying Li$^{*}$}\qquad  Shengjun Fan$^{**}$ \vspace{0.3cm} \\\textit{School of Mathematics, China University of Mining and Technology, Xuzhou 221116, PR China}}

\cortext[cor1]{E-mail address: lixinyingcumt@163.com}

\cortext[cor2]{Corresponding author. E-mail address: shengjunfan@cumt.edu.cn}

\vspace{0.2cm}
\begin{abstract}
In this paper, we are concerned with a multidimensional backward stochastic differential equation (BSDE) with a general random terminal time $\tau$, which may take values in $[0,+\infty)\cup \{+\infty\}$. Firstly, we establish an existence and uniqueness result for a weighted $L^p~(p>1)$ solution of the preceding BSDE with generator $g$ satisfying a stochastic monotonicity condition with general growth in the first unknown variable $y$ and a stochastic Lipschitz continuity condition in the second unknown variable $z$. Then, we derive an existence and uniqueness result for a  weighted $L^1$ solution of the preceding BSDE under an additional stochastic sub-linear growth condition in $z$. These results generalize the corresponding ones obtained in Li, Zhang and Fan [2025, SPA] to the $L^p~(p\geq 1)$ solution case. Finally, the corresponding comparison theorems for the weighted $L^p~(p\geq1)$ solutions are put forward and verified in the one-dimensional setting. In particular, we develop new ideas and systematic techniques in order to establish the above results.
\vspace{0.2cm}
\end{abstract}

\begin{keyword}
Backward stochastic differential equation \sep Existence and uniqueness \sep\\  \hspace*{1.80cm} Stochastic monotonicity condition \sep Stochastic sub-linear growth \sep Comparison theorem\vspace{0.1cm}

\MSC[2021] 60H10\vspace{0.2cm}
\end{keyword}

\end{frontmatter}
\vspace{-0.4cm}

\section{Introduction}
Let $k$ and $d$ be two positive integers, $(B_t)_{t\geq0}$ a standard $d$-dimensional Brownian motion defined on a complete probability space $(\Omega,\F,\mathbb{P})$ generating an augmented $\sigma$-algebra filtration $(\F_t)_{t\geq0}$, $\tau$ a general $(\F_t)$-stopping time taking values in $[0,+\infty)\cup \{+\infty\}$ called the random terminal time (or the random time horizon) and $\F:=\F_\tau$. In this paper, we consider the backward stochastic differential equation (BSDE) of the following form:
\begin{align}\label{BSDE1.1}
  Y_t=\xi+\int_t^\tau g(s,Y_s,Z_s){\rm d}s-\int_t^\tau Z_s{\rm d}B_s, \ \ t\in[0,\tau],
\end{align}
where $\xi$ is an $\F_\tau$-measurable $k$-dimensional random vector called the terminal value, and the random function
$$
g(\omega,t,y,z): \Omega\times[0,\tau]\times \R^k\times \R^{k\times d}\mapsto \R^k
$$
is $(\F_t)$-progressively measurable for each $(y,z)$ called the generator of BSDE \eqref{BSDE1.1}. We usually denote BSDE with parameters $(\xi,\tau,g)$ as BSDE $(\xi,\tau,g)$. As indicated in \cite{Li2024}, it is more natural to study the type of BSDE \eqref{BSDE1.1} than the following type: for each $T>0$,
\begin{equation}\label{BSDE1.1*}
\left\{
\begin{array}{l}
\Dis y_t=y_{T\wedge\tau}+\int_{t\wedge\tau}^{T\wedge\tau} g(s,y_s,z_s){\rm d}s-\int_{t\wedge\tau}^{T\wedge\tau} z_s{\rm d}B_s, \ \ t\in [0,T];\vspace{0.2cm}\\
\Dis y_t=\xi\ \ {\rm on\  the\ set\ of}\ \{t\geq \tau\},
\end{array}
\right.
\end{equation}
and it also seems to be more convenient in order to unify those works on BSDEs with the bounded terminal time, unbounded terminal time and infinite terminal time.

It is well known that \cite{PardouxPeng1990SCL} initially proved an existence and uniqueness result for an $L^2$ solution of BSDE \eqref{BSDE1.1} with the constant terminal time where the generator $g$ is uniformly Lipschitz continuous in $(y,z)$. Since then, BSDEs have undergone extensive research, and revealed numerous applications in mathematical finance, stochastic control, partial differential equations, and other related fields. Interested readers are referred to \cite{KarouiEl1997}, \cite{ZChenBWang2000JAMS}, \cite{Yong2006}, \cite{DelbaenTang2010}, \cite{Bahlali2015}, \cite{JiRonglinShiXueJun2019}, \cite{Tian2023SIAM} and \cite{Li2024} for more details. However, in order to align with more practical applications, many scholars are devoted to relaxing the uniform Lipschitz continuity condition on the generator and improving the constant time horizon to the infinite time horizon case and further generalizing it to the random time horizon case. Meanwhile, the existence and uniqueness, and the comparison theorems for $L^p~(p\geq1)$ solutions have also been extensively studied.

In particular, the uniform Lipschitz continuity condition of the generator $g$ in $y$ used in \cite{PardouxPeng1990SCL} was successfully relaxed to the monotonicity condition together with a general growth condition in \cite{Pardoux1999}. By using the convolution approaching idea, the a priori estimate technique and the truncation argument, \cite{Briand2003SPA} further weakened this general growth condition and generalized the result in \cite{Pardoux1999} to the $L^p~(p\geq1)$ solution case, where the generator $g$ has a kind of sub-linear growth in $z$ in the case of $p=1$. These results were extended to the case of infinite terminal time by a better truncation argument in \cite{Xiao2015} where the generator $g$ satisfies a time-varying monotonicity condition in $y$, i.e., there exists a deterministic integrable function $\mu(\cdot)$ such that for each $y_1, \ y_2\in\R^k$ and $z\in\R^{k\times d}$,
$$\left\langle y_1-y_2,g(\omega,t,y_1,z)-g(\omega,t,y_2,z)\right\rangle\leq \mu(t)|y_1-y_2|^2, \ \ t\in[0,\tau],$$
and a time-varying Lipschitz continuity condition in $z$, i.e., there exists a deterministic square-integrable function $\nu(\cdot)$ such that
for each $y\in\R^k$ and $z_1,z_2\in\R^{k\times d}$,
$$\left|g(\omega,t,y,z_1)-g(\omega,t,y,z_2)\right|\leq \nu(t)|z_1-z_2|, \ \ t\in[0,\tau],$$
as well as an additional time-varying sub-linear growth condition in $z$ for the case of $p=1$. We also would like to mention that the $L^p~(p>1)$ solution for the type of BSDE \eqref{BSDE1.1*} was studied by \cite{PardouxandRascanu(2014)} and \cite{O2020}.
Very recently, under the conditions of stochastic monotonicity in $y$ and stochastic Lipschitz continuity in $z$ for the generator $g$, that is, the aforementioned two conditions hold with two nonnegative processes $\mu_\cdot$ and $\nu_\cdot$ satisfying the condition of $\int_0^{\tau(\omega)} (\mu_t(\omega)+\nu_t^2(\omega)){\rm d}t<+\infty$ rather than the deterministic functions $\mu(\cdot)$ and $\nu(\cdot)$, \cite{Li2024} used a delicate truncation argument to establish a general existence and uniqueness result for a weighted $L^2$ solution of BSDE \eqref{BSDE1.1} with a general random terminal time $\tau$, lying in a weighted $L^2$ space with a weighted factor $e^{\int_0^t (\beta \mu_s(\omega)+\frac{\rho}{2}\nu_s^2(\omega)){\rm d}s}$ for any given $\beta\geq 1$ and $\rho>1$, which unifies and strengthens many corresponding existing results on the $L^2$ solution. Then, the following question naturally arises: can the result of \cite{Li2024} be generalized to the $L^p~(p\geq1)$ solution case that includes all above-mentioned results on the $L^p$ solution as its special cases? If the answer is affirmative, what are the weighted factors for the case of $p>1$ and $p=1$, and what additional conditions are required in the case of $p=1$? These questions are answered in this paper.

Let us proceed with presenting some existing results that are closely related to our work. By subdividing the time interval via stopping times and using Picard's iteration method, \cite{LiT2019} investigated existence and uniqueness of the $L^2$ solution for infinite time horizon BSDE \eqref{BSDE1.1} under a weak stochastic-monotonicity condition of the generator $g$ in $y$ together with the stochastic Lipschitz condition in $z$, which was extended to the $L^p~(p>1)$ solution case in \cite{LiFan2023CSTM} by an ingenious truncation argument. Readers are also referred to \cite{Liu2020} for the case where the generator $g$ satisfies a stochastic Lipschitz condition in $(y,z)$ with stochastic coefficients $\mu_\cdot$ and $\nu_\cdot$. It should be noted that the stochastic coefficients $\mu_\cdot$ and $\nu_\cdot$ in all the above-mentioned literature in this paragraph are required to satisfy $\int_0^{\tau(\omega)} (\mu_t(\omega)+\nu^2_t(\omega)){\rm d}t\leq M$ for a constant $M>0$. Furthermore, BSDE \eqref{BSDE1.1} with constant terminal time $T$ and stochastic coefficients $\mu_\cdot$ and $\nu_\cdot$ with $\int_0^{T}(\mu_t(\omega)+\nu_t^2(\omega)){\rm d}t$ having a certain exponential moment was addressed by \cite{Yong2006} and \cite{Bahlali2015}. On the other hand, \cite{BenderKohlmann2000} and \cite{Li2023} studied existence and uniqueness for the $L^2$ solutions of random time horizon BSDE with stochastic Lipschitz continuity generators, and the $L^p~(p>1)$ solution case was tackled in \cite{WangRanChen2007} by the truncation technique and the convolution approaching method and in \cite{O2020} by means of the random time-change idea. We would like to mention that the terminal time $\tau$ needs to be bounded and the stochastic coefficients $\mu_\cdot$ and $\nu_\cdot$ are imposed the condition of $\int_0^{\tau(\omega)} (\mu_t(\omega)+\nu_t^2(\omega)){\rm d}t\leq M$ for a constant $M>0$ when $p\in(1,2)$ in \cite{WangRanChen2007}, and an additional restrictive condition for $\mu_\cdot$ and $\nu_\cdot$ is required in \cite{O2020}. Readers are also referred to \cite{Owo(2017)} for the study on reflected BSDEs with stochastic Lipschitz continuity generators. Finally, some existence and uniqueness results on the $L^1$ solution of BSDE \eqref{BSDE1.1} with finite terminal time or infinite terminal time were established in \cite{FanLiuSPL} and \cite{SFan2016SPA}. However, to the best of our knowledge, there are no papers that study the $L^1$ solution of BSDE \eqref{BSDE1.1} with a general random terminal time $\tau$ and unbounded stochastic coefficients $\mu_\cdot$ and $\nu_\cdot$.

Enlightened by these results mentioned above, we first establish a general existence and uniqueness result for the weighted $L^p~(p>1)$ solution of BSDE \eqref{BSDE1.1} with a general random time horizon $\tau$, where the generator $g$ satisfies the stochastic monotonicity condition together with a general growth condition in $y$ and the stochastic Lipschitz continuity condition in $z$ with unbounded stochastic coefficients $\mu_\cdot$ and $\nu_\cdot$ (see \cref{thm:3.4} for more details). It should be emphasized that we don't impose any restriction of finite moment on the stochastic coefficients $\mu_\cdot$ and $\nu_\cdot$. After that, we study existence and uniqueness for the weighted $L^1$ solution with an additional assumption that the generator has a stochastic sub-linear growth in $z$ with stochastic coefficient $\gamma_\cdot$ (see \cref{thm:4.1} for more details). We employ $e^{\int_0^t\left(\beta\mu_s+\frac{\rho}{2[(p-1)\wedge1]} \nu_s^2\right){\rm d}s}$ for any given $\beta\geq1$ and $\rho>1$ as the weighted factor when $p>1$, and suppose an additional condition of $\int_0^{\tau(\omega)}\nu_t^2{\rm d}t\leq M$ for a constant $M>0$ and use $e^{\int_0^t\beta\mu_s{\rm d}s}$ as the weighted factor when $p=1$, which are both different from those used in \cite{WangRanChen2007}, \cite{Owo(2017)}, \cite{O2020} and \cite{Li2024}. Analogous to results in \cite{Li2024}, our results unify and strengthen some corresponding results on the $L^p~(p\geq1)$ solution of BSDE \eqref{BSDE1.1} with stochastic Lipschitz/monotonicity generators and BSDE \eqref{BSDE1.1} with deterministic Lipschitz/monotonicity generators. In particular, the integrability conditions for $g(\cdot,0,0)$ used in \cite{BenderKohlmann2000}, \cite{Bahlali2004}, \cite{WangRanChen2007} and \cite{O2020} are all stronger than our \ref{A:H3} (see \cref{rmk:h1} for more details), and the additional integrability conditions for stochastic coefficients $\mu_\cdot$ and $\nu_\cdot$ in \cite{WangRanChen2007} and \cite{O2020} are not required any more. Due to the more general integrability on the stochastic coefficients and the more general weighted spaces with different weighted factors, we systematically employ the methods utilized in \cite{Briand2003SPA}, \cite{Xiao2015}, \cite{LiT2019} and \cite{Li2024} and develop some innovative ideas to address the new challenges that arise naturally. For a comprehensive understanding, we introduce some remarks and examples in Sections 3 and 4 to compare our results with some closely related works and illustrate the innovation of our work. Finally, in the one-dimensional setting, the corresponding comparison theorems for the weighted $L^p~(p>1)$ solutions and the weighted $L^1$ solutions are put forward and verified in Theorems \ref{thm:com1} and \ref{thm:com2}. We overcome some new difficulties caused by the unbounded stochastic coefficients, see \cref{rmk5.3} for more details.

The remainder of this paper is structured as follows. Section 2 presents two a priori estimates-Propositions \ref{pro:1.1} and \ref{pro:1.2}, which are used repeatedly in the subsequent sections. In Section 3, we state \cref{thm:3.4} on the weighted $L^p~(p>1)$ solution for BSDE \eqref{BSDE1.1}, provide its proof and present some examples and remarks. In Section 4, the existence and uniqueness for the $L^1$ solution of BSDE \eqref{BSDE1.1} is established in \cref{thm:4.1} and some examples and remarks are also given. Comparison theorems for the weighted $L^p~(p\geq1)$ solutions of one-dimensional BSDEs are finally addressed in Section 5.

\section{Preliminaries and a priori estimates}
\setcounter{equation}{0}
In this section, we introduce some notations to be used later. In this paper, all equalities and inequalities between random elements are understood to hold $\mathbb{P}$-a.s.. Let $|\cdot|$ denote the usual Euclidean norm and $\langle x,y\rangle$ the usual Euclidean inner product of $x,y\in \R^k$. Denote $\R_+:=[0,+\infty)$, $a\wedge b:=\min\{a,b\}$, ${\bf 1}_A=1$ when $x\in A$ otherwise 0 and $\hat{y}:=\frac{y}{|y|}{\bf 1}_{|y|\neq0}$ for each $y\in \R^k$.

For any $(\F_t)$-progressively measurable nonnegative process $(a_t)_{t\in[0,\tau]}$, we introduce the following weighted $L^p~(p>0)$ spaces, which are just the spaces used in \cite{Li2024} in the case of $p=2$.

$\bullet$  $L_\tau^p(a_\cdot;\R^k)$ denotes the set of all $\xi$ that are $\F_\tau$-measurable $\R^k$-valued random vectors with
$$\|\xi\|_{p;a_\cdot}:=\left(\E\left[e^{p \int_0^\tau a_s{\rm d}s}|\xi|^p\right]\right)^{1\wedge{\frac{1}{p}}}<+\infty.$$

$\bullet$  $S_\tau^p(a_\cdot;\R^k)$ denotes  the set of all $(Y_t)_{t\in[0,\tau]}$ that are  $(\F_t)$-adapted, $\R^k$-valued and continuous processes  with
$$\|Y_\cdot\|_{p;a_\cdot,c}:=\left\{\E\left[\sup_{t\in[0,\tau]}\left(e^{p \int_0^t a_r{\rm d}r}|Y_t|^p\right)\right]\right\}^{1\wedge{\frac{1}{p}}}<+\infty.$$

$\bullet$  $M_\tau^p(a_\cdot;\R^{k\times d})$ denotes the set of all $(Z_t)_{t\in[0,\tau]}$ that are  $(\F_t)$-progressively measurable $\R^{k\times d}$-valued processes with
$$\|Z_\cdot\|_{p;a_\cdot}:=\left\{\E\left[\left(\int_0^\tau e^{2 \int_0^s a_r{\rm d}r}|Z_s|^2{\rm d}s\right)^{\frac{p}{2}}\right]\right\}^{1\wedge{\frac{1}{p}}}<+\infty.$$
Furthermore, it is clear that when $p\geq1$,
$$H_\tau^p(a_\cdot;\R^{k}\times\R^{k\times d}):=S_\tau^p(a_\cdot;\R^k)\times M_\tau^p(a_\cdot;\R^{k\times d})$$
is a Banach space with the norm
$$\|(Y_\cdot,Z_\cdot)\|_{p;a_\cdot}:=\|Y_\cdot\|_{p;a_\cdot,c}
+\|Z_\cdot\|_{p;a_\cdot}.$$
Let us recall that an $(\F_t)$-adapted, $\R^k$-valued and continuous process $(Y_t)_{t\in[0,\tau]}$ belongs to the class (D) if the family of variables $\{|Y_{\tau'}|:\tau'\in\Sigma_\tau\}$ is uniformly integrable, here and hereafter $\Sigma_\tau$ represents the set of all $(\F_t)$-measurable stopping times valued in $[0,\tau]$. For a process $(Y_t)_{t\in[0,\tau]}$ belonging to the class (D), we define
$$||Y||_1:=\sup\{\E[|Y_{\tau'}|]:\tau'\in\Sigma_\tau\}.$$

For convenience, we introduce the following definitions concerning solutions of BSDE (1.1).

\begin{dfn}
A solution of BSDE \eqref{BSDE1.1} is a pair of $(\F_t)$-progressively measurable processes $(y_t,z_t)_{t\in[0,\tau]}$ with values in $\R^k\times \R^{k\times d}$, such that $\mathbb{P}-a.s.$, $\int_0^\tau |g(t,y_t,z_t)|{\rm d}t<+\infty$, $\int_0^\tau |z_t|^2{\rm d}t<+\infty$, and \eqref{BSDE1.1} holds for each $t\in[0,\tau]$.
\end{dfn}

\begin{dfn}
Assume that $(y_t,z_t)_{t\in[0,\tau]}$ is a solution to BSDE \eqref{BSDE1.1} and $(a_t)_{t\in[0,\tau]}$ is an $(\F_t)$-progressively measurable nonnegative process. If $(y_t,z_t)_{t\in[0,\tau]}\in S_\tau^p(a_\cdot;\R^k)\times M_\tau^p(a_\cdot;\R^{k\times d})$ for some $p>1$, then $(y_t,z_t)_{t\in[0,\tau]}$ is called a weighted $L^p$ solution of BSDE \eqref{BSDE1.1} in $H_\tau^p(a_\cdot;\R^{k}\times\R^{k\times d})$. If $(y_t,z_t)_{t\in[0,\tau]}\in\bigcap_{\theta\in(0,1)}\left[S_\tau^\theta(a_\cdot;\R^k)\times M_\tau^\theta(a_\cdot;\R^{k\times d})\right]$ and $(e^{\int_0^ta_s{\rm d}s}y_t)_{t\in[0,\tau]}$ belongs to the class (D), then $(y_t,z_t)_{t\in[0,\tau]}$ is called a weighted $L^1$ solution of BSDE \eqref{BSDE1.1} in the space of $\bigcap_{\theta\in(0,1)}\left[S_\tau^\theta(a_\cdot;\R^k)\times M_\tau^\theta(a_\cdot;\R^{k\times d})\right]$.
\end{dfn}

The following Propositions \ref{pro:1.1} and \ref{pro:1.2} are two important a priori estimates concerning the weighted $L^p~(p\geq1)$ solution of BSDE \eqref{BSDE1.1}, which will be frequently used later. In stating them, the following assumption on the generator $g$ will be used.

\begin{enumerate}
\renewcommand{\theenumi}{(A)}
\renewcommand{\labelenumi}{\theenumi}
\item\label{A:A}
For each $(y,z)\in\R^k\times\R^{k\times{d}}$, it holds that
$$\left<\hat{y},g(\omega,t,y,z)\right>\leq u_{t}(\omega)|y|+v_t(\omega)|z|+f_{t}(\omega), \ \ t\in[0,\tau],$$
where $(u_t)_{t\in [0,\tau]}$ and $(v_t)_{t\in [0,\tau]}$ are nonnegative $(\F_t)$-progressively measurable processes such that $$\int_{0}^{\tau}\overline{a}_t{\rm d}t<+\infty \ \ {\rm with} \ \ \overline{a}_t:=\beta u_t+\frac{\rho}{2[(p-1)\wedge1]} v_t^2{\bf 1}_{p>1}, \ \ \beta\geq1 \ \ {\rm and} \ \ \rho>1,$$
and $f_\cdot$ is an $(\F_t)$-progressively measurable nonnegative process such that $$\E\left[\left(\int_{0}^{\tau}e^{ \int_{0}^{t}\overline{a}_r{\rm d}r}f_t{\rm d}t\right)^p\right]<+\infty.$$
\end{enumerate}

\begin{pro}\label{pro:1.1}
Assume that $p>0$, the generator $g$ satisfies \ref{A:A} and $(Y_t,Z_t)_{t\in[0,\tau]}$ is a solution of BSDE \eqref{BSDE1.1} such that $Y_\cdot\in S_\tau^p(\overline{a}_\cdot;\R^k)$. If $p>1$, then $Z_\cdot\in M_\tau^p(\overline{a}_\cdot;\R^{k\times d})$ and there exists a constant $C_{p,\rho}$ depending only on $p$ and $\rho$ such that for each $0\leq r\leq t<+\infty$,
\begin{align}\label{pro:1.1-1}
\begin{split}
&\E\left[\left(\int_{t\wedge\tau}^{\tau}e^{2\int_{0}^{s}\overline{a}_r{\rm d}r}|Z_s|^2{\rm d}s\right)^{\frac{p}{2}}\bigg|\F_{r\wedge\tau}\right]\\
&\ \ \leq
C_{p,\rho}\left(\E\left[\sup_{s\in[t\wedge\tau,\tau]}\left(e^{p{\int_{0}^{s}\overline{a}_r{\rm d}r}}|Y_s|^p\right)\bigg|\F_{r\wedge\tau}\right]+ \E\left[\left(\int_{t\wedge\tau}^{\tau}e^{{\int_{0}^{s}\overline{a}_r{\rm d}r}}f_s{\rm d}s\right)^p\bigg|\F_{r\wedge\tau}\right]\right).
\end{split}
\end{align}
If $p\in(0,1]$ and $\int_{0}^{\tau}v_t^2{\rm d}t\leq M$ for a constant $M>0$, then $Z_\cdot\in M_\tau^p(\beta u_\cdot;\R^{k\times d})$ and there exists a constant $C_{p,\rho,M}$ depending only on $p$, $\rho$ and $M$ such that for each $0\leq r\leq t<+\infty$,
\begin{align}\label{pro:1.1-2}
\begin{split}
&\E\left[\left(\int_{t\wedge\tau}^{\tau}e^{2\beta\int_{0}^{s}u_r{\rm d}r}|Z_s|^2{\rm d}s\right)^{\frac{p}{2}}\bigg|\F_{r\wedge\tau}\right]\\
&\ \ \leq
C_{p,\rho,M}\left(\E\left[\sup_{s\in[t\wedge\tau,\tau]}\left(e^{p\beta{\int_{0}^{s}u_r{\rm d}r}}|Y_s|^p\right)\bigg|\F_{r\wedge\tau}\right]+ \E\left[\left(\int_{t\wedge\tau}^{\tau}e^{\beta{\int_{0}^{s}u_r{\rm d}r}}f_s{\rm d}s\right)^p\bigg|\F_{r\wedge\tau}\right]\right).
\end{split}
\end{align}
\end{pro}

\begin{proof}
For each integer $n\geq1$, introduce the following $(\F_t)$-stopping time
$$\tau_{n}:=\inf \left\{t\geq0: \int_{0}^{t}e^{2\int_{0}^{s}\overline{a}_r{\rm d}r}|Z_{s}|^{2} {\rm d}s \geq n\right\} \wedge \tau,$$
with convention that $\inf \emptyset=+\infty$.
Applying It\^{o}'s formula to $e^{2\int_{0}^{t}\overline{a}_s{\rm d}s}|Y_t|^2$ yields that
\begin{align}\label{2.4}
\begin{split}
&e^{2\int_{0}^{t\wedge \tau_n}\overline{a}_s{\rm d}s}|Y_{t\wedge \tau_n}|^2+\int_{t\wedge \tau_n}^{\tau_n} e^{2\int_{0}^{s}\overline{a}_r{\rm d}r}|Z_s|^2{\rm d}s+2\int_{t\wedge \tau_n}^{\tau_n} e^{2\int_{0}^{s}\overline{a}_r{\rm d}r}\overline{a}_s|Y_s|^2{\rm d}s\\
&\ \ =e^{2\int_{0}^{\tau_n}\overline{a}_s{\rm d}s}|Y_{\tau_n}|^2+2\int_{t\wedge \tau_n}^{\tau_n} e^{2\int_{0}^{s}\overline{a}_r{\rm d}r}\langle Y_s,g(s,Y_s,Z_s)\rangle{\rm d}s-2\int_{t\wedge \tau_n}^{\tau_n} e^{2\int_{0}^{s}\overline{a}_r{\rm d}r}\langle Y_s,Z_s{\rm d}B_s\rangle, \ \ t\geq0.
\end{split}
\end{align}
In light of \ref{A:A} and $2ab\leq \rho a^2+\frac{1}{\rho}b^2$, we get that 
\begin{align}\label{2.2}
\begin{split}
2\left<Y_s,g(s,Y_s,Z_s)\right>
\leq& 2u_{s}|Y_s|^2+2v_s|Y_s||Z_s|+2|Y_s|f_{s}\\
\leq& 2u_{s}|Y_s|^2+\frac{\rho}{(p-1)\wedge1} v_{s}^2|Y_s|^2+\frac{(p-1)\wedge1}{\rho}|Z_s|^2+2|Y_s|f_{s}, \ \ s\in[0,\tau_{n}].
\end{split}
\end{align}
It follows from \eqref{2.4}, \eqref{2.2}, $\beta\geq1$ and $\rho>1$ that for each $t\geq0$,
\begin{align}\label{1.3}
\begin{split}
&\left(1-\frac{(p-1)\wedge1}{\rho}\right)\int_{t\wedge \tau_n}^{\tau_n}e^{2\int_{0}^{s}\overline{a}_r{\rm d}r}|Z_s|^2{\rm d}s\\
&\ \ \leq2\sup_{s\in[{t\wedge \tau_n},\tau_n]}\left(e^{2\int_{0}^{s}\overline{a}_r{\rm d}r}|Y_s|^2\right)+\left(\int_{t\wedge \tau_n}^{\tau_n}e^{\int_{0}^{s}\overline{a}_r{\rm d}r}f_s{\rm d}s\right)^2+2\left|\int_{t\wedge \tau_n}^{\tau_n}e^{2\int_{0}^{s}\overline{a}_r{\rm d}r}\left<Y_s,Z_s{\rm d}B_s\right>\right|.
\end{split}
\end{align}
From \eqref{1.3} and the inequality $(a+b)^{p / 2} \leq 2^{p}\left(a^{p / 2}+b^{p / 2}\right)$, we deduce that for some constant $c_{p,\rho}>0$ depending only on $p$ and $\rho$, we have
\begin{align}\label{2.3}
\begin{split}
\left(\int_{t\wedge \tau_n}^{\tau_n}e^{2\int_{0}^{s}\overline{a}_r{\rm d}r}|Z_s|^2{\rm d}s\right)^\frac{p}{2}&\leq c_{p,\rho}\left[\sup_{s\in[{t\wedge \tau_n},\tau_n]}\left(e^{p\int_{0}^{s}\overline{a}_r{\rm d}r}|Y_s|^p\right)+\left(\int_{t\wedge \tau_n}^{\tau_n}e^{\int_{0}^{s}\overline{a}_r{\rm d}r}f_s{\rm d}s\right)^{p}\right]\\
&\ \ +c_{p,\rho}\left|\int_{t\wedge \tau_n}^{\tau_n}e^{2\int_{0}^{s}\overline{a}_r{\rm d}r}\left<Y_s,Z_s{\rm d}B_s\right>\right|^{\frac{p}{2}}, \ \ t\geq0.
\end{split}
\end{align}
Furthermore, by the conditional Burkholder--Davis--Gundy (BDG for short) inequality in \cite[Proposition 2.5]{HuLiWen2025} and Young's inequality we deduce that for each $0\leq r\leq t<+\infty$ and $n\geq m \geq1$,
\begin{align}\label{11.6}
\begin{split}
&c_{p,\rho}\E\left[\left|\int_{t\wedge\tau_n}^{\tau_n}e^{2\int_{0}^{s}\overline{a}_r{\rm d}r}\left<Y_s,Z_s{\rm d}B_s\right>\right|^{\frac{p}{2}}\bigg|\F_{r\wedge\tau_m}\right]\\
&\ \ \leq d_{p,\rho}\E\left[\left(\int_{t\wedge\tau_n}^{\tau_n}e^{4\int_{0}^{s}\overline{a}_r{\rm d}r}|Y_s|^2|Z_s|^2{\rm d}s\right)^{\frac{p}{4}}\bigg|\F_{r\wedge\tau_m}\right]\\
&\ \ \leq{\frac{{d^2_{p,\rho}}}{2}}\E\left[\sup_{s\in[t\wedge\tau_n,\tau_n]}\left(e^{p\int_{0}^{s}\overline{a}_r{\rm d}r}|Y_s|^p\right)\bigg|\F_{r\wedge\tau_m}\right]+\frac{1}{2}\E\left[\left(\int_{t\wedge\tau_n}^{\tau_n}e^{2\int_{0}^{s}\overline{a}_r{\rm d}r}|Z_s|^2{\rm d}s\right)^{\frac{p}{2}}\bigg|\F_{r\wedge\tau_m}\right],
\end{split}
\end{align}
where $d_{p,\rho}>0$ is a constant depending only on $p$ and $\rho$. Finally, in light of \eqref{11.6}, taking the  conditional mathematical expectation with respect to $\F_{r\wedge\tau_m}$ on both sides of \eqref{2.3}, letting $n\rightarrow +\infty$ and using Fatou's lemma and Lebesgue's dominated convergence theorem  yield that for each $0\leq r\leq t<+\infty$ and $m\geq1$,
\begin{align*}
\begin{split}
&\E\left[\left(\int_{t\wedge{\tau}}^{\tau}e^{2\int_{0}^{s}\overline{a}_r{\rm d}r}|Z_s|^2{\rm d}s\right)^{\frac{p}{2}}\bigg|\F_{r\wedge{\tau_m}}\right]\\
&\ \ \leq
C_{p,\rho}\left(\E\left[\sup_{s\in[t\wedge{\tau},{\tau}]}\left(e^{p{\int_{0}^{s}\overline{a}_r{\rm d}r}}|Y_s|^p\right)\bigg|\F_{r\wedge{\tau_m}}\right]+ \E\left[\left(\int_{t\wedge{\tau}}^{{\tau}}e^{{\int_{0}^{s}\overline{a}_r{\rm d}r}}f_s{\rm d}s\right)^p\bigg|\F_{r\wedge{\tau_m}}\right]\right).
\end{split}
\end{align*}
Thus, the desired inequality \eqref{pro:1.1-1} follows by sending $m\rightarrow\infty$ and using the martingale convergence theorem (see Corollary A.9 in Appendix C of \cite{Oksendal2005}) on both sides of the last inequality.

Finally, similar to the proof of \eqref{pro:1.1-1}, the desired inequality \eqref{pro:1.1-2} can be obtained by applying It\^{o}'s formula to $e^{2\beta\int_{0}^{t}u_s{\rm d}s}|Y_t|^2$, repeating the above steps and replacing \eqref{2.2} and \eqref{1.3} with the following \eqref{2.2-1} and \eqref{1.3-1} respectively:
\begin{align}\label{2.2-1}
\begin{split}
2\left<Y_s,g(s,Y_s,Z_s)\right>
\leq& 2u_{s}|Y_s|^2+2v_s|Y_s||Z_s|+2|Y_s|f_{s}\\
\leq& 2u_{s}|Y_s|^2+\rho v_{s}^2|Y_s|^2+\frac{1}{\rho}|Z_s|^2+2|Y_s|f_{s}, \ \ s\in[0,\tau_{n}],
\end{split}
\end{align}
and
\begin{align}\label{1.3-1}
\begin{split}
&\left(1-\frac{1}{\rho}\right)\int_{t\wedge \tau_n}^{\tau_n}e^{2\beta\int_{0}^{s}u_r{\rm d}r}|Z_s|^2{\rm d}s\\
&\ \ \leq(2+\rho M)\sup_{s\in[{t\wedge \tau_n},\tau]}\left(e^{2\beta\int_{0}^{s}u_r{\rm d}r}|Y_s|^2\right)+\left(\int_{t\wedge \tau_n}^{\tau_n}e^{\beta\int_{0}^{s}u_r{\rm d}r}f_s{\rm d}s\right)^2\\
&\ \ \ \ +2\left|\int_{t\wedge \tau_n}^{\tau_n}e^{2\beta\int_{0}^{s}u_r{\rm d}r}\left<Y_s,Z_s{\rm d}B_s\right>\right|.
\end{split}
\end{align}
The proof of \cref{pro:1.1} is then complete.
\end{proof}

\begin{pro}\label{pro:1.2}
Assume that $p>1$, the generator $g$ satisfies \ref{A:A}, the terminal value $\xi\in L_\tau^p(\overline{a}_\cdot;\R^k)$, and $(Y_t,Z_t)_{t\in[0,\tau]}$ is a solution of BSDE \eqref{BSDE1.1} such that $Y_\cdot\in S_\tau^p(\overline{a}_\cdot;\R^k)$. Then $Z_\cdot\in M_\tau^p(\overline{a}_\cdot;\R^{k\times d})$ and there exists a non-negative constant $K_{p,\rho}$ depending only on $p$ and $\rho$ such that for each $0\leq r\leq t<+\infty$,
\begin{align}\label{006*}
\begin{split}
&\E\left[\sup_{s\in[t\wedge\tau,\tau]}\left(e^{p\int_{0}^{s}\overline{a}_r{\rm d}r}|Y_s|^p\right)\bigg|\F_{r\wedge\tau}\right]+\E\left[\left(\int_{t\wedge\tau}^{\tau}e^{2\int_{0}^{s}\overline{a}_r{\rm d}r}|Z_s|^2{\rm d}s\right)^{\frac{p}{2}}\bigg|\F_{r\wedge\tau}\right]\\[2mm]
&\ \  \leq K_{p,\rho}\left(\E\left[e^{p{\int_{0}^{\tau}\overline{a}_s{\rm d}s}}|\xi|^p\bigg|\F_{r\wedge\tau}\right]+\E\left[\left(
\int_{t\wedge\tau}^{\tau}e^{\int_{0}^{s}\overline{a}_r{\rm d}r}f_s{\rm d}s\right)^p\bigg|\F_{r\wedge\tau}\right]\right).
\end{split}
\end{align}
\end{pro}

\begin{proof}
By Corollary 3.5 in \cite{WangRanChen2007} and assumption \ref{A:A}, we obtain
\begin{align}\label{111.6}
\begin{split}
&e^{p\int_{0}^{t\wedge\tau}\overline{a}_s{\rm d}s}|Y_{t\wedge\tau}|^p+c(p)\int_{t\wedge\tau}^{\tau}e^{p\int_{0}^{s}\overline{a}_r{\rm d}r}|Y_s|^{p-2}{\bf 1}_{|Y_s|\neq0}|Z_s|^2{\rm d}s+p\int_{t\wedge\tau}^{\tau} e^{p\int_{0}^{s}\overline{a}_r{\rm d}r}\overline{a}_s|Y_s|^{p}{\rm d}s\\
&\ \  \leq  e^{p\int_{0}^{\tau}\overline{a}_s{\rm d}s}|\xi|^p+p\int_{t\wedge\tau}^{\tau}e^{p\int_{0}^{s}\overline{a}_r{\rm d}r}\left(u_s|Y_s|^p+v_s|Y_s|^{p-1}|Z_s|+|Y_s|^{p-1}f_s\right){\rm d}s\\
&\ \  \ \  -p\int_{t\wedge\tau}^{\tau}e^{p\int_{0}^{s}\overline{a}_r{\rm d}r}|Y_s|^{p-1}\langle \hat{Y}_{s}, Z_{s} {\rm d} B_{s}\rangle, \ \ t\geq0,
\end{split}
\end{align}
where $c(p):=\frac{p[(p-1)\wedge1]}{2}$  and we recall $\hat{Y}_s:=\frac{Y_s}{|Y_s|}{\bf 1}_{|Y_s|\neq0}$.

It can be verified that $\{M_t=\int_{0}^{t}e^{p\int_{0}^{s}\overline{a}_r{\rm d}r}|Y_s|^{p-1}\langle \hat{Y}_{s}, Z_{s} {\rm d} B_{s}\rangle\}_{t\in[0,\tau]}$ is a uniformly integrable martingale. Indeed, by the BDG inequality, Young's inequality and \eqref{pro:1.1-1} of \cref{pro:1.1}, we know that
\begin{align}\label{00.4}
\begin{split}
&\E\left[\sup_{t\in[0,\tau]}\bigg|\int_{0}^{t}e^{p\int_{0}^{s}\overline{a}_r{\rm d}r}|Y_s|^{p-1}\langle \hat{Y}_{s}, Z_{s} {\rm d} B_{s}\rangle\bigg|\right]\\
&\ \  \leq 2\sqrt{2}\E\left[\left(\int_{0}^{\tau}e^{2p\int_{0}^{s}\overline{a}_r{\rm d}r}|Y_s|^{2p-2}|Z_{s}|^2 {\rm d}{s}\right)^{\frac{1}{2}}\right]\\
&\ \  \leq
2\sqrt{2}\E\left[\sup_{s\in[0,\tau]}\left(e^{(p-1)\int_{0}^{s}\overline{a}_r{\rm d}r}|Y_s|^{p-1}\right)\left(\int_{0}^{\tau}e^{2\int_{0}^{s}\overline{a}_r{\rm d}r}|Z_s|^2{\rm d}{s}\right)^{\frac{1}{2}}\right]\\
&\ \  \leq \frac{2\sqrt{2}(p-1)}{p}\E\left[\sup_{s\in[0,\tau]}\left(e^{p\int_{0}^{s}\overline{a}_r{\rm d}r}|Y_s|^{p}\right)\right]+\frac{2\sqrt{2}}{p}\E\left[\left(\int_{0}^{\tau}e^{2\int_{0}^{s}\overline{a}_r{\rm d}r}|Z_s|^2{\rm d}{s}\right)^{\frac{p}{2}}\right]<+\infty.
\end{split}
\end{align}
Moreover, in light of assumption \ref{A:A}, $Y_\cdot\in S_\tau^p(\overline{a}_\cdot;\R^k)$ and \eqref{pro:1.1-1} of \cref{pro:1.1}, by Young's inequality and H\"{o}lder's inequality we deduce that
\begin{align}\label{uyvz}
\begin{split}
&\int_0^{\tau}e^{p\int_{0}^{s}\overline{a}_r{\rm d}r}\left(u_s|Y_s|^p+v_s|Y_s|^{p-1}|Z_s|+|Y_s|^{p-1}f_s\right){\rm d}s\\
&\ \ \leq  \sup\limits_{s\in [0,\tau]}\left(e^{p\int_{0}^{s}\overline{a}_r{\rm d}r}|Y_s|^p\right)\int_0^\tau u_s{\rm d}s\\
&\ \ \ \ \ \ +\frac{p-1}{p}\sup\limits_{s\in [0,\tau]}\left(e^{p\int_{0}^{s}\overline{a}_r{\rm d}r}|Y_s|^p\right)+\frac{1}{p}\left(\int_0^\tau v_s^2{\rm d}s \right)^{\frac{p}{2}}\left(\int_0^\tau e^{2\int_{0}^{s}\overline{a}_r{\rm d}r}|Z_s|^2{\rm d}s \right)^{\frac{p}{2}}\\
&\ \ \ \ \ \ +\frac{p-1}{p}\sup\limits_{s\in [0,\tau]}\left(e^{p\int_{0}^{s}\overline{a}_r{\rm d}r}|Y_s|^p\right)+\frac{1}{p}\left(\int_0^\tau e^{\int_{0}^{s}\overline{a}_r{\rm d}r} f_s{\rm d}s\right)^p\\
&\ \ <+\infty.
\end{split}
\end{align}
Combining \eqref{111.6}-\eqref{uyvz} yields that for each $t\geq 0$,
\begin{align}\label{uyvz-1}
\begin{split}
\int_{t\wedge\tau}^{\tau}e^{p\int_{0}^{s}\overline{a}_r{\rm d}r}|Y_s|^{p-2}{\bf 1}_{|Y_s|\neq0}|Z_s|^2{\rm d}s<+\infty.
\end{split}
\end{align}
Furthermore, using Young's inequality we obtain
\begin{align}\label{1.7}
\begin{split}
&pe^{p\int_{0}^{s}\overline{a}_r{\rm d}r}v_s|Y_s|^{p-1}|Z_s|\\
&\ \  = p\left(\sqrt{\frac{\rho}{(p-1)\wedge1}}e^{\frac{p}{2}\int_{0}^{s}\overline{a}_r{\rm d}r}v_s|Y_s|^{\frac{p}{2}}\right)\left(\sqrt{\frac{(p-1)\wedge1}{\rho}}e^{\frac{p}{2}\int_{0}^{s}\overline{a}_r{\rm d}r}|Y_s|^{\frac{p}{2}-1}{\bf 1}_{|Y_s|\neq0}|Z_s|\right)\\
&\ \  \leq \frac{p\rho}{2[(p-1)\wedge1]}e^{p\int_{0}^{s}\overline{a}_r{\rm d}r}v_s^2|Y_s|^{p}+\frac{c(p)}{\rho}e^{p\int_{0}^{s}\overline{a}_r{\rm d}r}|Y_s|^{p-2}{\bf 1}_{|Y_s|\neq0}|Z_s|^2, \ \ s\in[0,\tau].
\end{split}
\end{align}
It follows from \eqref{111.6}, \eqref{uyvz-1} and \eqref{1.7} along with the facts of $\beta\geq1$ and $\rho>1$ that
\begin{align}\label{1.8-1}
\begin{split}
&e^{p\int_{0}^{t\wedge\tau}\overline{a}_s{\rm d}s}|Y_{t\wedge\tau}|^p+\left(1-\frac{1}{\rho}\right)c(p)\int_{t\wedge\tau}^{\tau}e^{p\int_{0}^{s}\overline{a}_r{\rm d}r}|Y_s|^{p-2}{\bf 1}_{|Y_s|\neq0}|Z_s|^2{\rm d}s\\
&\ \  \leq X_{\tau}^t-p\int_{t\wedge\tau}^{\tau}e^{p\int_{0}^{s}\overline{a}_r{\rm d}r}|Y_s|^{p-1}\langle \hat{Y}_{s}, Z_{s} {\rm d} B_{s}\rangle, \ \ t\geq0,
\end{split}
\end{align}
where $$X_{\tau}^t:=e^{p\int_{0}^{\tau}\overline{a}_s{\rm d}s}|\xi|^p+p\int_{t\wedge\tau}^{\tau}e^{p\int_{0}^{s}\overline{a}_r{\rm d}r}|Y_s|^{p-1}f_s{\rm d}s.$$
Then, in light of \eqref{00.4}, taking the conditional mathematical expectation on both sides of \eqref{1.8-1} yields that for each $0\leq r\leq t<+\infty$,
\begin{align}\label{00-4}
\begin{split}
\left(1-\frac{1}{{\rho}}\right)c(p)\E\left[\int_{t\wedge\tau}^{\tau}e^{p\int_{0}^{s}\overline{a}_r{\rm d}r}|Y_s|^{p-2}{\bf 1}_{|Y_s|\neq0}|Z_s|^2{\rm d}s\bigg|\F_{r\wedge\tau}\right]\leq\E\left[X_\tau^t|\F_{r\wedge\tau}\right].
\end{split}
\end{align}

On the other hand, by the conditional BDG inequality in \cite[Proposition 2.5]{HuLiWen2025} and $2ab\leq a^2+b^2$, we also deduce that for each $0\leq r\leq t<+\infty$,
\begin{align*}
\begin{split}
&p\E\left[\sup_{s\in[t\wedge\tau,\tau]}\bigg|\int_{s}^{\tau}e^{p\int_{0}^{s}\overline{a}_r{\rm d}r}|Y_s|^{p-1}\langle \hat{Y}_{s}, Z_{s} {\rm d} B_{s}\rangle\bigg|\bigg|\F_{r\wedge\tau}\right]\\
&\ \  \leq  2\sqrt{2} p \E\left[\left(\int_{t\wedge\tau}^{\tau}e^{2p\int_{0}^{s}\overline{a}_r{\rm d}r}|Y_s|^{2p-2}|Z_{s}|^2 {\rm d}{s}\right)^{\frac{1}{2}}\bigg|\F_{r\wedge\tau}\right]\\
&\ \  \leq 2\sqrt{2} p \E\left[\sup_{s\in[t\wedge\tau,\tau]}\left(e^{\frac{p}{2}\int_{0}^{s}\overline{a}_r{\rm d}r}|Y_s|^{\frac{p}{2}}\right)\left(\int_{t\wedge\tau}^{\tau}e^{p\int_{0}^{s}\overline{a}_r{\rm d}r}|Y_s|^{p-2}{\bf 1}_{|Y_s|\neq0}|Z_s|^2{\rm d}s\right)^{\frac{1}{2}}\bigg|\F_{r\wedge\tau}\right]\\
&\ \  \leq \frac{1}{2}\E\left[\sup_{s\in[t\wedge\tau,\tau]}\left(e^{p\int_{0}^{s}\overline{a}_r{\rm d}r}|Y_s|^{p}\right)\bigg|\F_{r\wedge\tau}\right] +4p^2 \E\left[\int_{t\wedge\tau}^{\tau}e^{p\int_{0}^{s}\overline{a}_r{\rm d}r}|Y_s|^{p-2}{\bf 1}_{|Y_s|\neq0}|Z_s|^2{\rm d}s\bigg|\F_{r\wedge\tau}\right].
\end{split}
\end{align*}
Then, in light of the last inequality and \eqref{1.8-1}, we have for each $0\leq r\leq t<+\infty$,
\begin{align}\label{1.9}
\begin{split}
&\frac{1}{2}\E\left[\sup_{s\in[t\wedge\tau,\tau]}\left(e^{p\int_{0}^{s}\overline{a}_r{\rm d}r}|Y_s|^p\right)\bigg|\F_{r\wedge\tau}\right]\\
&\ \  \leq \E[X_{\tau}^t|\F_{r\wedge\tau}]+4p^2\E\left[\int_{t\wedge\tau}^{\tau}e^{p\int_{0}^{s}\overline{a}_r{\rm d}r}|Y_s|^{p-2}{\bf 1}_{|Y_s|\neq0}|Z_s|^2{\rm d}s\bigg|\F_{r\wedge\tau}\right].
\end{split}
\end{align}
Combining \eqref{00-4} and \eqref{1.9} we get the existence of a constant $\overline{k}_{p,\rho}>0$ depending only on $p$ and $\rho$ such that for each $0\leq r\leq t<+\infty$,
\begin{align}\label{1.10}
\begin{split}
\E\left[\sup_{s\in[t\wedge\tau,\tau]}\left(e^{p\int_{0}^{s}\overline{a}_r{\rm d}r}|Y_s|^p\right)\bigg|\F_{r\wedge\tau}\right]\leq \overline{k}_{p,\rho}\E[X_{\tau}^t|\F_{r\wedge\tau}].
\end{split}
\end{align}
Finally, in light of Young's inequality, for each $0\leq r\leq t<+\infty$ we have
\begin{align}\label{1.11}
\begin{split}
&p\overline{k}_{p,\rho}\E\left[\int_{t\wedge\tau}^{\tau}e^{p\int_{0}^{s}\overline{a}_r{\rm d}r}|Y_s|^{p-1}f_s{\rm d}s\bigg|\F_{r\wedge\tau}\right]\\
&\ \ \leq p\overline{k}_{p,\rho}\E\left[\sup_{s\in[t\wedge\tau,\tau]}\left(e^{(p-1)\int_{0}^{s}\overline{a}_r{\rm d}r}|Y_s|^{p-1}\int_{t\wedge\tau}^{\tau}e^{\int_{0}^{s}\overline{a}_r{\rm d}r}f_s{\rm d}s\right)\bigg|\F_{r\wedge\tau}\right]\\
&\ \ \leq \frac{1}{2}\E\left[\sup_{s\in[t\wedge\tau,\tau]}\left(e^{p\int_{0}^{s}\overline{a}_r{\rm d}r}|Y_s|^{p}\right)\bigg|\F_{r\wedge\tau}\right]+\tilde{k}_{p,\rho}\E\left[\left(\int_{t\wedge\tau}^{\tau}e^{\int_{0}^{s}\overline{a}_r{\rm d}r}f_s{\rm d}s\right)^p\bigg|\F_{r\wedge\tau}\right],
\end{split}
\end{align}
where $\tilde{k}_{p,\rho}>0$ is a constant depending only on $p$ and $\rho$. Then, the desired assertion \eqref{006*} follows from \eqref{1.10}, \eqref{1.11} and \eqref{pro:1.1-1} of \cref{pro:1.1} along with the definition of $X_\tau^t$. The proof is complete.
\end{proof}

\section{Weighted $L^p~(p>1)$ solutions}
\setcounter{equation}{0}
In this section, let $p>1$ be a given constant. We shall put forward and prove an existence and uniqueness result for the weighted $L^p$ solution to BSDE \eqref{BSDE1.1} which extends the corresponding conclusion in \cite{Li2024} to the $L^p$ solution case. Furthermore, we will introduce several corollaries, remarks and examples to show that it also generalizes some existing results including Theorem 4.2 in \cite{Briand2003SPA}, Theorem 4.1 in \cite{WangRanChen2007}, Theorem 3.1 in \cite{Xiao2015} and so on.

In the rest of this paper, we always assume that $\beta$, $\rho$ and $M$ are given constants such that $\beta\geq1$, $\rho>1$ and $M>0$, and that  $(u_t)_{t\in [0,\tau]}$ and $(v_t)_{t\in [0,\tau]}$ are two given nonnegative $(\F_t)$-progressively measurable processes and $$a_t:=\beta\mu_t+\frac{\rho}{2[(p-1)\wedge1]} \nu_t^2,\ \ t\in [0,\tau]$$
satisfying
$$\int_{0}^{\tau}{a}_t{\rm d}t<+\infty.$$

\subsection{Statement of the existence and uniqueness result for the weighted $L^p~(p>1)$ solution}
Let us start with introducing the following assumptions on $g$.

\begin{enumerate}
\renewcommand{\theenumi}{(H1)}
\renewcommand{\labelenumi}{\theenumi}
\item\label{A:H1} $\E\left[\left(\int_0^\tau e^{\int_{0}^{s}a_r{\rm d}r}|g(s,0,0)|{\rm d}s\right)^p\right]<+\infty.$
\renewcommand{\theenumi}{(H2)}
\renewcommand{\labelenumi}{\theenumi}
\item\label{A:H2} ${\rm d}\mathbb{P}\times{\rm d} t-a.e.$, $g(\omega,t,\cdot,z)$ is continuous for each $z\in\R^{k\times d}$.
\renewcommand{\theenumi}{(H3)}
\renewcommand{\labelenumi}{\theenumi}
\item\label{A:H3} $g$ has a general growth in $y$, i.e., there exists an $(\F_t)$-progressively measurable nonnegative decreasing process $(\alpha_t)_{t\in[0,\tau]}$ taking values in $(0,1]$ { and satisfying $\essinf\limits_{t\in[0,\tau]}\alpha_t>0$} such that for each $r\in \R_+$, it holds that
$$\E\left[\int_0^\tau e^{\beta \int_{0}^{t}\mu_s{\rm d}s}\psi_{r}^{\alpha_\cdot}(t)dt\right]<+\infty$$
with
$$\psi_{r}^{\alpha_\cdot}(t):=\sup_{|y|\leq r\alpha_t}\left\{ \left|g(t,y,0)-g(t,0,0)\right|\right\}, \ \ t\in[0,\tau].$$
\renewcommand{\theenumi}{(H4)}
\renewcommand{\labelenumi}{\theenumi}
\item\label{A:H4} $g$ satisfies a stochastic monotonicity condition in $y$, i.e., for each $y_1, \ y_2\in\R^k$ and $z\in\R^{k\times d}$,
$$\left\langle y_1-y_2,~g(\omega,t,y_1,z)-g(\omega,t,y_2,z)\right\rangle\leq \mu_t(\omega)|y_1-y_2|^2, \ \ t\in[0,\tau].$$
\renewcommand{\theenumi}{(H5)}
\renewcommand{\labelenumi}{\theenumi}
\item\label{A:H5} $g$ satisfies a stochastic Lipschitz continuity condition in $z$, i.e., for each $y\in\R^k$ and $z_1,z_2\in\R^{k\times d}$,
$$\left|g(\omega,t,y,z_1)-g(\omega,t,y,z_2)\right|\leq \nu_t(\omega)|z_1-z_2|, \ \ t\in[0,\tau].$$
\end{enumerate}

We state the following theorem which is the main result of this subsection and establishes a general existence and
uniqueness result of the weighted $L^p$ solution.
\begin{thm}\label{thm:3.4}
Assume that $p>1$, $\xi\in L_\tau^p(a_\cdot;\R^k)$, and the generator $g$ satisfies assumptions \ref{A:H1}-\ref{A:H5}. Then, BSDE (1.1) admits a unique weighted $L^p$ solution $(Y_t,Z_t)_{t\in[0,\tau]}$ in  $H_\tau^p(a_\cdot;\R^{k}\times\R^{k\times d})$.
\end{thm}

\begin{rmk}\label{rmk3.3}
The above assumptions \ref{A:H1} with $p=2$, \ref{A:H2}-\ref{A:H5} are used in \cite{Li2024}, and $\mu_\cdot$ and $\nu_\cdot$ aren't imposed any restriction of finite moment. Additionally, based on the proof procedure of Corollary 3.2 in \cite{Li2024}, it is clear that BSDEs with stochastic Lipschitz generators satisfy assumptions \ref{A:H2}-\ref{A:H5}, and \cref{thm:3.4} can apply to them.
\end{rmk}

%
%


To facilitate a comprehensive comparison between \cref{thm:3.4} and related existing results, let us introduce the following stronger assumptions of the generator $g$ than \ref{A:H1} and \ref{A:H3}, and the following remark.
\begin{enumerate}
\renewcommand{\theenumi}{(H1a)}
\renewcommand{\labelenumi}{\theenumi}
\item\label{A:H1a} There exists $\delta>\left(\frac{2}{p-1}\vee3\right)$ and $\varepsilon>0$ such that $\mu_\cdot+\nu_\cdot^2\geq \varepsilon$ and
    $$\E\left[\left(\int_0^\tau e^{\delta\int_0^s(\mu_r+\nu_r^2){\rm d}r}\bigg|\frac{g(s,0,0)}{\sqrt{\mu_s+\nu_s^2}}\bigg|^2{\rm d}s\right)^{\frac{p}{2}}\right]<+\infty.$$
\renewcommand{\theenumi}{(H1b)}
\renewcommand{\labelenumi}{\theenumi}
\item\label{A:H1b} $\Dis \E\left[\int_0^\tau e^{p\theta\int_0^s (\mu_r+\frac{1}{2(p-1)}\nu^2_r) {\rm d}r}(\mu_s+\nu^2_s)^{-(p-1)}\left|g(s,0,0)\right|^p{\rm d}s\right]<+\infty$, where $\theta>1$ and $\mu_\cdot+\nu^2_\cdot\geq\varepsilon>0$.
    \renewcommand{\theenumi}{(H3a)}
\renewcommand{\labelenumi}{\theenumi}
\item\label{A:H3a}
For each $r\in \R_+$, it holds that
$$
\E\left[\left(\int_0^\tau e^{\int_{0}^{t}\mu_s{\rm d}s}\psi_{r}(t)dt\right)^p\right]<+\infty
$$
with
$$\psi_{r}(t):=\sup_{|y|\leq r}\left|g(t,y,0)-g(t,0,0)\right|, \ \ t\in[0,\tau].\vspace{-0.1cm}$$
\end{enumerate}
\vspace{1mm}

\begin{rmk}\label{rmk:h1}
We have the following remarks.

\vspace{1.5mm}
(i) \ref{A:H1a} comes from (A4) of \cite{WangRanChen2007} and (H3) of \cite{Owo(2017)}. If \ref{A:H1a} holds, then taking $\beta:=\frac{\delta+3}{4}>1$, ${\rho}:=\frac{(\delta+3)[(p-1)\wedge1]}{2}>1$, and using H\"{o}lder's inequality we have
\begin{align*}
\begin{split}
&\E\left[\left(\int_0^\tau e^{\int_0^s\left(\beta\mu_r+\frac{{\rho}}{2[(p-1)\wedge1]}\nu_r^2\right){\rm d}r}|g(s,0,0)|{\rm d}s\right)^p\right]\\
&\ \ =\E\left[\left(\int_0^\tau e^{\frac{\delta}{2}\int_0^s(\mu_r+\nu_r^2){\rm d}r}\bigg|\frac{g(s,0,0)}{\sqrt{\mu_s+\nu_s^2}}\bigg|
e^{-\frac{\delta-3}{4}\int_0^s(\mu_r+\nu_r^2){\rm d}r}\sqrt{\mu_s+\nu_s^2}{\rm d}s\right)^p\right]\\
&\ \ \leq\E\left[\left(\int_0^\tau e^{{\delta}\int_0^s(\mu_r+\nu_r^2){\rm d}r}\bigg|\frac{g(s,0,0)}{\sqrt{\mu_s+\nu_s^2}}\bigg|^2{\rm d}s\right)^\frac{p}{2}\right]
\E\left[\left(\int_0^\tau e^{-\frac{\delta-3}{2}\int_0^s(\mu_r
+\nu_r^2){\rm d}r}(\mu_s+\nu_s^2){\rm d}r\right)^\frac{p}{2}\right]\\
&\ \ <+\infty,
\end{split}
\end{align*}
which means that assumption \ref{A:H1} is true. Consequently, \ref{A:H1} is weaker than \ref{A:H1a} and \cref{thm:3.4} generalizes Theorem 4.1 of \cite{WangRanChen2007} where $\tau$ is a bounded stopping time and $\int_0^\tau (\mu_t+\nu_t^2){\rm d}t\leq M$ for $p\in(1,2)$ and Theorem 4.3 of \cite{Owo(2017)} with $g\equiv0$.

\vspace{1.5mm}
(ii) \ref{A:H1b} comes from 4 of Theorem 5.2 in \cite{O2020}. If \ref{A:H1b} holds for some $p\in(1,2)$, then taking $\beta:=\frac{3\theta+1}{4}>1$, ${\rho}:=\frac{(p-1)(1-\theta)+2\theta}{2}>1$, and using H\"{o}lder's inequality we have
\begin{align}
&\E\left[\left(\int_0^\tau e^{\int_0^s\left(\beta\mu_r+\frac{{\rho}}{2(p-1)}\nu_r^2\right){\rm d}r}|g(s,0,0)|{\rm d}s\right)^p\right]\nonumber\\
&\ \ =\E\left[\left(\int_0^\tau e^{\theta\int_0^s(\mu_r+\frac{\nu_r^2}{2(p-1)}){\rm d}r}\left(\mu_s+\nu_s^2\right)^{-\frac{p-1}{p}}|g(s,0,0)|
e^{-\frac{\theta-1}{4}\int_0^s(\mu_r+\nu_r^2){\rm d}r}\left(\mu_s+\nu_s^2\right)^{\frac{p-1}{p}}{\rm d}s\right)^p\right]\nonumber\\
&\ \ \leq\E\left[\int_0^\tau e^{p\theta\int_0^s(\mu_r+\frac{\nu_r^2}{2(p-1)}){\rm d}r}\left(\mu_s+\nu_s^2\right)^{-(1-p)}|g(s,0,0)|^p{\rm d}s\right]\nonumber\\
&\ \ \ \ \ \times\E\left[\left(\int_0^\tau e^{-\frac{p(\theta-1)}{4(p-1)}\int_0^s(\mu_r
+\nu_r^2){\rm d}r}(\mu_s+\nu_s^2){\rm d}s\right)^{p-1}\right]<+\infty.\nonumber
\end{align}
Thus assumption \ref{A:H1} holds for $p\in(1,2)$. Consequently, \cref{thm:3.4} can be compared with Theorem 5.2 in \cite{O2020}, where the type of BSDE \eqref{BSDE1.1*} was addressed.

\vspace{1.5mm}
(iii) \ref{A:H3a} was used in Theorem 5.30, Theorem 5.57, and Corollary 5.59 of \cite{PardouxandRascanu(2014)}. It is clear that \ref{A:H3a} can imply \ref{A:H3} with $\beta=1$ and $\alpha_\cdot\equiv 1$, while \ref{A:H3} can not imply \ref{A:H3a}. In fact, if the generator $g$ satisfies stochastic Lipschitz continuity condition in $(y,z)$, then it follows from Corollary 3.2 of \cite{Li2024} that this $g$ satisfies \ref{A:H3}, while it does not satisfy \ref{A:H3a} under the situation of
$$
\E\left[\left(\int_0^\tau e^{\int_{0}^{t}\mu_s{\rm d}s}\mu_t dt\right)^p\right]=+\infty.\vspace{0.2cm}
$$
\end{rmk}

Next, we introduce another stronger assumption \ref{A:H3b} of the generator $g$ than \ref{A:H3}, and give the following \cref{cor:h3s}.
\begin{enumerate}
\renewcommand{\theenumi}{(H3b)}
\renewcommand{\labelenumi}{\theenumi}
\item\label{A:H3b} There exists an $(\F_t)$-progressively measurable nonnegative process $(\tilde{\mu}_t)_{t\in[0,\tau]}$ { satisfying $\int_0^\tau \tilde{\mu}_t{\rm d}t<+\infty$} such that for each $y\in \R^k$,
    $$|g(t,y,0)-g(t,0,0)|\leq \tilde{\mu}_t\varphi(|y|), \ \ t\in[0,\tau],$$
    where $\varphi(\cdot):\R_+\rightarrow \R_+$ is a nondecreasing convex function with $\varphi(0)=0$.
\end{enumerate}

\begin{cor}\label{cor:h3s}
Assume that $p>1$, $\xi\in L_\tau^p(a_\cdot;\R^k)$, and the generator $g$ satisfies \ref{A:H1}, \ref{A:H2}, \ref{A:H3b}, \ref{A:H4} and \ref{A:H5}. Then, BSDE (1.1) admits a unique weighted $L^p$ solution $(Y_t,Z_t)_{t\in[0,\tau]}$ in  $H_\tau^p(a_\cdot;\R^{k}\times\R^{k\times d})$.
\end{cor}

\begin{proof}
It follows from (iv) of Remark 3.3 in \cite{Li2024} that \ref{A:H3b} is strictly stronger than \ref{A:H3}. Then \cref{cor:h3s} follows immediately from \cref{thm:3.4}.
\end{proof}

\begin{rmk}\label{rmk:h3s}
Assumption \ref{A:H3b} can be regarded as a generalization of the corresponding assumptions c(ii) in \cite{PardouxPeng1990SCL}, (H5') in \cite{Briand2003SPA} and (H3') in \cite{Xiao2015}. And, \ref{A:H3b} is easier to be verified than \ref{A:H3}, although \ref{A:H3} is strictly weaker than \ref{A:H3b}, and \ref{A:H3b} will be repeatedly utilized to provide some examples in the subsequent subsection.
\end{rmk}

We proceed to compare our results with some existing results. The following corollary is a direct consequence of \cref{thm:3.4}, and the following stronger assumptions of the generator $g$ than \ref{A:H1} and \ref{A:H3} will be used in it.

\begin{enumerate}
\renewcommand{\theenumi}{(H1c)}
\renewcommand{\labelenumi}{\theenumi}
\item\label{A:H1c} $\Dis \E\left[\left(\int_0^\tau |g(s,0,0)|{\rm d}s\right)^p\right]<+\infty.$
\renewcommand{\theenumi}{(H3c)}
\renewcommand{\labelenumi}{\theenumi}
\item\label{A:H3c} There exists an $(\F_t)$-progressively measurable non-increasing process $(\alpha_t)_{t\in[0,\tau]}$ taking values in $(0,1]$ { and satisfying $\essinf\limits_{t\in[0,\tau]}\alpha_t>0$} such that for each $r\in \R_+$, it holds that
$$\E\left[\int_0^\tau \psi_{r}^{\alpha_\cdot}(t)dt\right]<+\infty,$$
where $\psi_{r}^{\alpha_\cdot}(t)$ is defined in \ref{A:H3}.
\end{enumerate}

\begin{cor}\label{cor:3.4}
Assume that $p>1$, $\xi\in L_\tau^p(0;\R^k)$, and the generator $g$ satisfies \ref{A:H1c}, \ref{A:H2}, \ref{A:H3c}, \ref{A:H4} and \ref{A:H5} with
\begin{align*}
\int_0^\tau\left(\mu_t+\nu_t^2\right){\rm d}t\leq M.
\end{align*}
Then, BSDE (1.1) admits a unique weighted $L^p$ solution $(Y_t,Z_t)_{t\in[0,\tau]}$ in  $H_\tau^p(0;\R^{k}\times\R^{k\times d})$.
\end{cor}

\begin{rmk}\label{rmk:bound}
We make the following remarks.

\vspace{1.5mm}
(i) Due to the presence of $\alpha_\cdot$, assumption \ref{A:H3c} is strictly weaker than (H5) used in \cite{Briand2003SPA} and (H3) used in \cite{Xiao2015}, where $\alpha_\cdot\equiv1$. Consequently, \cref{cor:3.4} and \cref{thm:3.4} strengthen Theorem 3.1 in \cite{Xiao2015} where terminal time $\tau$ is finite or infinite and both $\mu_\cdot$ and $\nu_\cdot$ are deterministic functions, and Theorem 4.2 in \cite{Briand2003SPA} where terminal time $\tau$, $\mu_\cdot$ and $\nu_\cdot$ are all finite positive constants.

\vspace{1.5mm}
(ii) Under the condition that $\int_0^\tau(\mu_t+\nu_t^2){\rm d}t$ has a certain exponential moment, \cite{Yong2006} dealt with the $L^p$ solvability of linear BSDEs with stochastic coefficients $\mu_\cdot$ and $\nu_\cdot$, and \cite{Bahlali2015} studied $L^p$ solvability of BSDEs with super-linear growth generators associated with stochastic coefficients $\mu_\cdot$ and $\nu_\cdot$. Moreover, \cite{BriandandConfortola2008} investigated the $L^p$ solvability of BSDEs with stochastic coefficients $\mu_\cdot$ and $\nu_\cdot$ satisfying a certain integrability condition related to the bounded mean oscillation martingale. In comparison with these results, \cref{thm:3.4} does not impose any restriction of finite moment on the stochastic coefficients $\mu_\cdot$ and $\nu_\cdot$.

\vspace{1.5mm}
(iii) \cref{thm:3.4}, and Corollaries \ref{cor:h3s} and \ref{cor:3.4} generalize Theorem 3.1, (iv) of Remark 3.3 and Corollary 3.4 in \cite{Li2024} to the $L^p$ solution case, respectively.
\end{rmk}

\subsection{Examples}
In this subsection, we present several examples where Theorem \ref{thm:3.4} can be applied, but none of existing results, including those \cite{Briand2003SPA}, \cite{WangRanChen2007}, \cite{Xiao2015}, \cite{Bahlali2015} and \cite{O2020}, is applicable.

\begin{ex}\label{ex:1}
Let $p>1$, $\tau$ be a bounded stopping time, i.e., $\tau\leq T$ for a real $T>0$, and for each $y=(y_1,\cdots,y_k)\in \R^k$ and $z\in \R^{k\times d}$, let
$$g(t,y,z):=(g^1(t,y,z),\cdots,g^k(t,y,z)), \ t\in[0,\tau],$$
where for each $i=1,\cdots,k,$
$$g^i(t,y,z):=e^{-|B_t|^6y_i}+\sin |z|.$$
It is easy to verify that this generator $g$ satisfies \ref{A:H1c}, \ref{A:H2}, \ref{A:H4} and \ref{A:H5} with $|g(\cdot,0,0)|\equiv \sqrt{k}$, $\mu_\cdot\equiv 0$ and $\nu_\cdot\equiv 1$. By Lemma 3.6 and Example 3.7 in \cite{Li2024}, we can deduce that the generator $g$ also satisfies assumption \ref{A:H3c} with $$\alpha_t:=\frac{1}{1+\sup\limits_{0\leq s\leq t}|B_s|^6}, \ t\in[0,\tau].$$
Thus, according to \cref{cor:3.4}, we know that for each $\xi\in L_\tau^p(0;\R^k)$, BSDE \eqref{BSDE1.1} admits a unique weighted $L^p~(p>1)$ solution in  $H_\tau^p(0;\R^{k}\times\R^{k\times d})$. However, by a discussion similar to Example 3.7 presented in \cite{Li2024}, we deduce that for each $r>0$, taking $y=(-r,0,\cdots,0)$,
$$\psi_{r}(t):=\sup_{|y|\leq r}\left|g(t,y,0)-g(t,0,0)\right|\geq e^{r|B_t|^6}-1,$$
and then
$$
\E\left[\int_0^\tau \psi_{r}(t)dt\right]=+\infty.\vspace{0.2cm}
$$
Thus $g$ does not satisfy assumption \ref{A:H3c} with $\alpha_\cdot\equiv1$ used in \cite{Briand2003SPA} and \cite{Xiao2015}. Therefore, the above conclusion can not be obtained from Theorem 4.2 in \cite{Briand2003SPA}, Theorem 4.1 in \cite{WangRanChen2007} and Theorem 3.1 in \cite{Xiao2015}.
\end{ex}

\begin{ex}\label{ex:3.12}
Let $p>1$, $k=2$, and $\tau$ be a finite stopping time, i.e., $P(\tau<+\infty)=1$. Consider the following generator: for $(y,z)\in \R^2\times \R^{2\times d}$ with $y=(y_1,y_2)$,
$$g(t,y,z):=|B_t|^2\begin{bmatrix}
-y_1^3+y_2\\
-y_2^5-y_1\\
\end{bmatrix}+|B_t|^3\begin{bmatrix}
|z|\\
\sin|z|\\
\end{bmatrix}, \ \ t\in[0,\tau].$$
It is straightforward to verify that $g$ satisfies assumptions \ref{A:H1}, \ref{A:H2}, \ref{A:H3b}, \ref{A:H4}-\ref{A:H5} with
$$g(t,0,0)=0, \ \tilde{\mu}_t=\mu_t=|B_t|^2, \ \nu_t=|B_t|^3, \ \varphi(x)=\sqrt{2}(x^5+x^3+x),\ x\geq0,$$
and $a_t:=\beta|B_t|^2+\frac{\rho}{2[(p-1)\wedge1]}|B_t|^6$ satisfying $\int_{0}^{\tau}{a}_t{\rm d}t<+\infty$. Therefore, by \cref{cor:h3s} we can conclude that for $\xi:=e^{-\int_0^\tau a_t{\rm d}t}B_\tau\in L_\tau^p(a_\cdot;\R^k)$, BSDE \eqref{BSDE1.1} admits a unique weighted $L^p~(p>1)$ solution. Since $g$ has a polynomial growth in $y$ and $\E[e^{\varepsilon\int_0^\tau a_t{\rm d}t}]=+\infty$ for each $\varepsilon>0$, this conclusion cannot be obtained from \cite{Briand2003SPA}, \cite{WangRanChen2007}, \cite{Xiao2015} and \cite{Bahlali2015}.
\end{ex}

\begin{ex}\label{ex:3.11}
Let $p>1$, $k=1$, $\tau$ be a stopping time taking values in $[0,+\infty)\cup \{+\infty\}$ and $\sigma$ be a finite stopping time. For each $(y,z)\in\R\times\R^{d}$, let
$$g(t,y,z):={ (|B_t|^4{\bf 1}_{0\leq t\leq \sigma}+e^{-t})}(1-e^{y^+})+|B_t|{\bf 1}_{0\leq t\leq \sigma}|y|+\sqrt{|B_t|{\bf 1}_{0\leq t\leq \sigma}}\sin |z|, \ \ t\in[0,\tau].$$
It is not hard to verify that this generator $g$ satisfies assumptions \ref{A:H1}, \ref{A:H2}, \ref{A:H3b}, \ref{A:H4} and \ref{A:H5} with
$$g(t,0,0)\equiv0, \ \mu_t=|B_t|{\bf 1}_{0\leq t\leq \sigma}, \ \nu_t=\sqrt{|B_t|{\bf 1}_{0\leq t\leq \sigma}}, \ \tilde{\mu}_t={ (|B_t|+|B_t|^4){\bf 1}_{0\leq t\leq \sigma}+e^{-t}}, \ \varphi(x)=e^x+x+1, \ x\geq0,$$
and $a_t:=\left(\beta+\frac{\rho}{2[(p-1)\wedge1]}\right)|B_t|{\bf 1}_{0\leq t\leq \sigma}$ satisfying $\int_0^\tau a_t{\rm d}t<+\infty$ due to the fact that $\sigma$ is a finite stopping time. From \cref{thm:3.4} or \cref{cor:h3s} we derive that for each $\xi\in L_\tau^p(a_\cdot;\R^k)$, BSDE \eqref{BSDE1.1} has a unique weighted $L^p$ solution in $H_\tau^p(a_\cdot;\R\times\R^{1\times d})$. However, to the best of our knowledge, this conclusion cannot be obtained by any known results.
\end{ex}

\subsection{Proof of Theorem \ref{thm:3.4}}

With the help of two a priori estimates of Propositions \ref{pro:1.1} and \ref{pro:1.2}, in this subsection we give the detailed proof of \cref{thm:3.4}.

\begin{proof}[\bf Proof of Theorem \ref{thm:3.4}.]
Let us start by studying the uniqueness part. Assume that $(Y_\cdot,Z_\cdot)$ and $(Y'_\cdot,Z'_\cdot)$ are two weighted $L^p$ solutions of BSDE \eqref{BSDE1.1} in $H_\tau^p(a_\cdot;\R^{k}\times\R^{k\times d})$. Denote by $(\tilde{Y}_\cdot,\tilde{Z}_\cdot)$ the process $(Y_\cdot-Y'_\cdot,Z_\cdot-Z'_\cdot)$. It is obvious that $(\tilde{Y}_\cdot,\tilde{Z}_\cdot)$ is a weighted $L^p$ solution in $H_\tau^p(a_\cdot;\R^{k}\times\R^{k\times d})$ to the following BSDE:
\begin{align*}
\tilde{Y}_t=\int_t^\tau\tilde{g}(s,\tilde{Y}_s,\tilde{Z}_s){\rm d}s-\int_t^\tau\tilde{Z}_s{\rm d}B_s, \ \ t\in[0,\tau],
\end{align*}
where for each $(y,z)\in\R^k\times\R^{k\times{d}}$, $\tilde{g}(t,y,z):=g(t,y+Y_t',z+Z_t')-g(t,Y_t',Z_t'), \ \ t\in[0,\tau].$ It follows from that \ref{A:H4} and \ref{A:H5} of $g$ that $\tilde{g}$ satisfies assumption \ref{A:A} with $f_\cdot\equiv0$. Then, it is easy to verify that $(\tilde{Y}_\cdot,\tilde{Z}_\cdot)\equiv0$ by \eqref{006*} of \cref{pro:1.2}.

Let us turn to the existence part. For each $(x,y,z)\in\R^k\times\R^k\times\R^{k\times{d}}$, $r>0$ and $n\geq1$, let $q_r(x):=\frac{xr}{|x|\vee r}$, $\xi_n:=q_{n\tilde{\alpha}_\tau}(\xi)$ and
\begin{align*}
g_n(t,y,z):=g(t,y,z)-g(t,0,0)+q_{ne^{-t}\tilde{\alpha}_t}(g(t,0,0)), \ t\in[0,\tau]
\end{align*}
with $\tilde{\alpha}_t:=e^{-\int_{0}^{t}a_s{\rm d}s}$.
Then, for each $n\geq1$, we have
\begin{align*}
|\xi_n|\leq n\tilde{\alpha}_\tau, \ \ {\rm and} \ \  |g_n(t,0,0)|\leq ne^{-t}\tilde{\alpha}_t, \ \ t\in[0,\tau].
\end{align*}
Since $\xi\in L_\tau^p(a_\cdot;\R^k)$, by \ref{A:H1} and Lebesgue's dominated convergence theorem, we get
\begin{align}\label{1.14}
\lim\limits_{n\rightarrow\infty} \E\left[e^{p\int_{0}^{s}a_r{\rm d}r}|\xi|^p{\bf 1}_{|\xi|>n\tilde{\alpha}_\tau}+\left(\int_0^\tau e^{\int_{0}^{s}a_r{\rm d}r}|g(s,0,0)|{\bf 1}_{|g(s,0,0)|>ne^{-s}\tilde{\alpha}_s}{\rm d}s\right)^p\right]=0.
\end{align}
Furthermore, it is easy to verify that $g_n$ satisfies assumptions \ref{A:H1} with $p=2$ and \ref{A:H2}-\ref{A:H5} and that $\xi_n\in L_\tau^2(\beta\mu_\cdot+\frac{\rho}{2}\nu_\cdot^2;\R^k)$. It then follows from Theorem 3.1 of \cite{Li2024} that for each $n\geq1$, BSDE $(\xi_n,\tau,g_n)$ has a unique weighted $L^2$ solution $(Y_t^n,Z_t^n)_{t\in[0,\tau]}$ in $H_\tau^2(\beta\mu_\cdot+\frac{\rho}{2}\nu_\cdot^2;\R^{k}\times\R^{k\times d})$. On the other hand, for each $n\geq1$ and $(y,z)\in \R^k\times\R^{k\times d}$, by \ref{A:H4} and \ref{A:H5} we have
\begin{align*}
\begin{split}
\left<\hat{y},g_n(t,y,z)\right>&\leq \mu_t|y|+\nu_t|z|+\bigg|q_{ne^{-t}\tilde{\alpha}_t}(g(t,0,0))\bigg| \leq \mu_t|y|+\nu_t|z|+ne^{-t}\tilde{\alpha}_t, \ \ t\in[0,\tau],
\end{split}
\end{align*}
which implies that $g_n$ satisfies assumption \ref{A:A} with $u_t:=\mu_t$, $v_t:=\nu_t$, $f_t:=ne^{-t}\tilde{\alpha}_t$ and $\overline{a}_t:=\beta\mu_t+\frac{\rho}{2}\nu_t^2$. It then follows from \cref{pro:1.2} with $p=2$ and $r=t$ that there exists a constant $K_\rho>0$ such that for each $t\geq0$,
\begin{align*}
\begin{split}
&e^{2\int_{0}^{t\wedge\tau}(\beta\mu_s+\frac{\rho}{2}\nu_s^2){\rm d}s}|Y_{t\wedge\tau}^n|^2\\
&\ \ \ \leq K_\rho\left(\E\left[e^{2\int_{0}^{\tau}(\beta\mu_s+\frac{\rho}{2}\nu_s^2){\rm d}s}|\xi_n|^2\bigg|\F_{t\wedge\tau}\right]
+\E\left[\left(\int_{t\wedge\tau}^{\tau}e^{\int_{0}^{s}(\beta\mu_r+\frac{\rho}{2}\nu_r^2){\rm d}r}ne^{-s}\tilde{\alpha}_s{\rm d}s\right)^2\bigg|\F_{t\wedge\tau}\right]\right).
\end{split}
\end{align*}
Multiplying $e^{\int_{0}^{t\wedge\tau}(\frac{\rho}{(p-1)\wedge1}-\rho)\nu_s^2{\rm d}s}$ on both sides of the last inequality yields that for each $t\geq0$,
\begin{align*}
\begin{split}
e^{2\int_{0}^{t\wedge\tau}a_s{\rm d}s}|Y_{t\wedge\tau}^n|^2
\leq& K_\rho\left(\E\left[e^{2\int_{0}^{\tau}a_s{\rm d}s}|\xi_n|^2\bigg|\F_{t\wedge\tau}\right]+\E\left[\left(\int_{0}^{\tau}e^{\int_{0}^{s}a_r{\rm d}r}ne^{-s}\tilde{\alpha}_s{\rm d}s\right)^2\bigg|\F_{t\wedge\tau}\right]\right)\\
\leq&2K_\rho n^2,
\end{split}
\end{align*}
which implies that for each $n\geq1$, $(e^{\int_{0}^{t}a_s{\rm d}s}|Y_t^n|)_{t\in[0,\tau]}$ is a bounded process and
$$
\E\left[\sup_{t\in[0,\tau]}\left(e^{p\int_{0}^{t}a_s{\rm d}s}|Y_t^n|^p\right)\right]<+\infty.
$$
Then $(Y_t^n)_{t\in[0,\tau]}$ belongs to $S_\tau^p(a_\cdot;\R^k)$, and it follows from $\xi_n\in L_\tau^2(a_\cdot;\R^k)$ and \eqref{pro:1.1-1} of \cref{pro:1.1} that for each $n\geq1$, $(Z_t^n)_{t\in[0,\tau]}$ belongs to $M_\tau^p(a_\cdot;\R^{k\times d})$.

Next, we prove that $\{(Y_\cdot^n,Z_\cdot^n)\}^{+\infty}_{n=1}$ is a Cauchy sequence in $H_\tau^p(a_\cdot;\R^{k}\times\R^{k\times d})$. For each pair of integers $n,i\geq 1$, let
$$
\hat{\xi}^{n,i}:=\xi_{n+i}-\xi_n, \  \ \hat Y_\cdot^{n,i}:=Y_\cdot^{n+i}-Y_\cdot^n, \ \ \hat Z_\cdot^{n,i}:=Z_\cdot^{n+i}-Z_\cdot^n.
$$
Then
\begin{align*}
\hat Y_t^{n,i}=\hat{\xi}^{n,i}+\int_t^\tau\hat{g}^{n,i}(s,\hat Y_s^{n,i},\hat Z_s^{n,i}){\rm d}s-\int_t^\tau\hat Z_s^{n,i}{\rm d}B_s, \ \ t\in[0,\tau],
\end{align*}
where for each $(y,z)\in \R^k\times\R^{k\times d}$, $$\hat{g}^{n,i}(s,y,z):=g_{n+i}(s,y+Y_s^n,z+Z_s^n)-g_n(s,Y_s^n,Z_s^n), \ \ s\in[0,\tau].$$
It follows from \ref{A:H4} and \ref{A:H5} of $g$ that for each $(y,z)\in\R^k\times\R^{k\times d}$, and $n,i\geq1$,
\begin{align*}
\begin{split}
\left<\hat{y},\hat{g}^{n,i}(t,y,z)\right>&=\left<\hat{y},g_{n+i}(t,y+Y_s^n,z+Z_s^n)-g_{n+i}(t,Y_t^n,Z_t^n)+g_{n+i}(t,Y_t^n,Z_t^n)-g_n(t,Y_t^n,Z_t^n)\right>\\
&\leq \mu_t|y|+\nu_t|z|+|g(t,0,0)|{\bf 1}_{|g(t,0,0)|>ne^{-t}\tilde{\alpha}_t}, \ \ t\in[0,\tau].
\end{split}
\end{align*}
So $\hat{g}^{n,i}$ satisfies assumption \ref{A:A} with $u_t:=\mu_t$, $v_t:=\nu_t$, $\overline{a}_t:=a_t$ and
$f_t:=|g(t,0,0)|{\bf 1}_{|g(t,0,0)|>ne^{-t}\tilde{\alpha}_t}.$
According to \cref{pro:1.2} with $t=r=0$, we can conclude that there exists a non-negative constant $C_{p,\rho}>0$ depending only on $p$ and $\rho$ such that for each $n, i\geq1$,
\begin{align}\label{3.6}
\begin{split}
&\|\hat{Y}_\cdot^{n,i}\|_{p;a_\cdot,c}^p+\|\hat{Z}_\cdot^{n,i}\|_{p;a_\cdot}^p\\
&\ \  \leq  C_{p,\rho}\E\left[e^{{p\int_{0}^{\tau}a_s{\rm d}s}}|\xi|^p{\bf 1}_{|\xi|>n\tilde{\alpha}_\tau}+\left(\int_0^\tau e^{\int_{0}^{s}a_r{\rm d}r}|g(t,0,0)|{\bf 1}_{|g(s,0,0)|>ne^{-s}\tilde{\alpha}_s}{\rm d}s\right)^p\right].
\end{split}
\end{align}
Furthermore, taking first the supremum with respect to $i$ and then taking the limit with respect to $n$ on both sides of \eqref{3.6}, and combining \eqref{1.14} and Lebesgue's dominated convergence theorem yield that $\{(Y_\cdot^n,Z_\cdot^n)\}^{+\infty}_{n=1}$ is a Cauchy sequence in $H_\tau^p(a_\cdot;\R^{k}\times\R^{k\times d})$.

Finally, we denote by $(Y_t,Z_t)_{t\in[0,\tau]}$ the limit of Cauchy sequence $\{(Y_t^n,Z_t^n)_{t\in[0,\tau]}\}_{n=1}^\infty$ in $H_\tau^p(a_\cdot;\R^{k}\times\R^{k\times d})$, and pass to the limit under the uniform convergence in probability (ucp) for BSDE $(\xi_n,\tau,g_n)$, in light of the definition of $\xi_n$ and $g_n$ along with assumptions of $g$, to see that $(Y_t,Z_t)_{t\in[0,\tau]}$ is the desired weighted $L^p$ solution of BSDE \eqref{BSDE1.1} in $H_\tau^p(a_\cdot;\R^{k}\times\R^{k\times d})$. The existence part is also proved, and the whole proof is complete.
\end{proof}

\begin{rmk}\label{rmk3.2}
Since we consider existence and uniqueness of the solution in the weighted $L^p$ space with a weighted factor $e^{\int_0^t a_s{\rm d}s}$ for $a_\cdot:=\beta\mu_\cdot+\frac{\rho}{2[(p-1)\wedge1]} \nu_\cdot^2$, what kind of truncation should be used in order to apply Theorem 3.1 in \cite{Li2024}, and how to prove that the weighted $L^2$ solution for truncated BSDE $(\xi_n,\tau,g_n)$ also belongs to the weighted $L^p$ space are two crucial points in the proof of \cref{thm:3.4}.
\end{rmk}

\section{Weighted $L^1$-solutions}
\setcounter{equation}{0}
In this section, we will put forward and prove an existence and uniqueness result for the weighted $L^1$ solution of BSDE \eqref{BSDE1.1}.

\subsection{Statement of the existence and uniqueness result for the weighted $L^1$ solution}
Recalling $\beta$ and $M$ are two given constants such that $\beta\geq1$ and $M>0$, and $(\mu_t)_{t\in [0,\tau]}$ is a given nonnegative process satisfying $\int_0^\tau \mu_t{\rm d}t<+\infty$. We will state a general existence and uniqueness result for the weighted $L^1$ solution of BSDE \eqref{BSDE1.1}. In addition to \ref{A:H2}-\ref{A:H5}, the generator $g$ needs to satisfy the following assumptions \ref{A:H1'} and \ref{A:H6}, which is a stochastic sub-linear growth condition in $z$.
\begin{enumerate}
\renewcommand{\theenumi}{(H1')}
\renewcommand{\labelenumi}{\theenumi}
\item\label{A:H1'}  $\E\left[\int_0^\tau e^{\beta\int_{0}^{s}\mu_r{\rm d}r}|g(s,0,0)|{\rm d}s\right]<+\infty.$
\renewcommand{\theenumi}{(H6)}
\renewcommand{\labelenumi}{\theenumi}
\item\label{A:H6}  There exists an $l\in(0,1)$ such that for each $y\in\R^k$ and $z \in\R^{k\times d}$,
    $$|g(t,y,z)-g(t,y,0)|\leq\gamma_t(g_t^1+g_t^2+|y|+|z|)^l, \ \ t\in[0,\tau],$$
    where $(\gamma_t)_{t\in[0,\tau]}$, $(g_t^1)_{t\in[0,\tau]}$ and $(g_t^2)_{t\in[0,\tau]}$ are three $(\F_t)$-progressively measurable nonnegative processes satisfying
    $$\int_0^\tau \left(e^{(1-l)\beta\int_0^s \mu_r{\rm d}r}\gamma_s\right)^{\frac{1}{1-l}}{\rm d}s+\int_0^\tau e^{(1-l)\beta\int_0^s \mu_r{\rm d}r}\gamma_s{\rm d}s\leq M,$$
    $$\E\left[\int_0^\tau e^{\beta\int_0^s \mu_r{\rm d}r}g_s^1{\rm d}s\right]<+\infty \ \ {\rm and} \ \ \E\left[\sup\limits_{t\in[0,\tau]}\left(e^{\beta\int_0^t \mu_r{\rm d}r}|g_t^2|\right)\right]<+\infty.$$
\end{enumerate}

\begin{thm}\label{thm:4.1}
Assume that the generator $g$ satisfies assumptions \ref{A:H1'} and \ref{A:H2}-\ref{A:H6}, and $\int_0^\tau \nu_t^2{\rm d}t\leq M$. Then for each $\xi\in L_\tau^1(\beta\mu_\cdot;\R^k)$, BSDE \eqref{BSDE1.1} has a unique weighted $L^1$ solution $(y_t,z_t)_{t\in[0,\tau]}$ in the space of $\bigcap_{\theta\in(0,1)}\left[S_\tau^\theta(\beta\mu_\cdot;\R^k)\times M_\tau^\theta(\beta\mu_\cdot;\R^{k\times d})\right]$.
\end{thm}

\begin{rmk}\label{rmk-4}
We have the following remarks.

\vspace{1.5mm}
(i) Assumption \ref{A:H6} is strictly weaker than (H7) in \cite{Briand2003SPA} in which the terminal time $\tau$ is a finite real and all of $\gamma_\cdot$, $\mu_\cdot$ and $\nu_\cdot$ are finite positive constants, and (H6) of \cite{Xiao2015} where the terminal time $\tau$ is finite or infinite real and all of $\gamma_\cdot$, $\mu_\cdot$ and $\nu_\cdot$ are deterministic non-negative real functions. \cref{thm:4.1} establishes existence and uniqueness of the weighted $L^1$ solution of BSDE \eqref{BSDE1.1} with a general random terminal time $\tau$ and unbounded stochastic coefficients $\mu_\cdot$ and $\nu_\cdot$, which unifies and strengthens Theorem 6.3 of \cite{Briand2003SPA} and Theorem 4.1 of \cite{Xiao2015}.

\vspace{1.5mm}
(ii) In \cref{thm:4.1}, we impose a stronger condition of $\int_0^\tau \nu_t^2{\rm d}t\leq M$ than $\int_0^{\tau}\nu_t^2{\rm d}t< +\infty$ in order to use the method of dividing time interval to construct a contraction mapping from the weighted $L^p~(p>1)$ space to itself, see the second step in the proof of the existence part of \cref{thm:4.1} in subsection 4.3 for more details. Thus, the solution appearing in \cref{thm:4.1} belongs to the weighted $L^1$ space with a weighted factor $e^{\beta\int_{0}^{t }\mu_s{\rm d}s}$.

\vspace{1.5mm}
(iii) Note that $1<\frac{2}{2-l}<\frac{1}{1-l}$ for $l\in(0,1)$. For each $x\geq0$, we have $x^{\frac{2}{2-l}}\leq x+x^{\frac{1}{1-l}}$, then it follows from assumption \ref{A:H6} that
\begin{align}\label{gama}
\begin{split}
\int_0^\tau \left(e^{(1-l)\beta\int_0^s \mu_r{\rm d}r}\gamma_s\right)^{\frac{2}{2-l}}{\rm d}s\leq M.
\end{split}
\end{align}

\vspace{1.5mm}
(iv) If assumption \ref{A:H6} is replaced with the following inequality: for each $y\in\R^k$ and $z \in\R^{k\times d}$,
$$|g(t,y,z)-g(t,y,0)|\leq\gamma_t|z|^l, \ \ t\in[0,\tau]$$
with $\gamma_\cdot$ satisfying \eqref{gama}, then \cref{thm:4.1} still holds. This assertion can be verified by the proof procedure of \cref{thm:4.1} postponed in subsection 4.3. It is more convenient to give examples by this assertion.
\end{rmk}

As indicated in \cref{rmk:h3s} of Section 3, although assumption \ref{A:H3b} is strictly stronger than \ref{A:H3}, it is easier to be verified. Then we introduce the following useful corollary.

\begin{cor}\label{cor:h3s1}
Assume that the generator $g$ satisfies assumptions \ref{A:H1'}, \ref{A:H2}, \ref{A:H3b} and \ref{A:H4}-\ref{A:H6}, and $\int_0^\tau \nu_t^2{\rm d}t\leq M$. Then for each $\xi\in L_\tau^1(\beta\mu_\cdot;\R^k)$, BSDE \eqref{BSDE1.1} admits a unique weighted $L^1$ solution $(y_t,z_t)_{t\in[0,\tau]}$ in the space of $\bigcap_{\theta\in(0,1)}\left[S_\tau^\theta(\beta\mu_\cdot;\R^k)\times M_\tau^\theta(\beta\mu_\cdot;\R^{k\times d})\right]$.
\end{cor}


Next, we present another corollary of \cref{thm:4.1} that is easily verified and its proof is omitted here. The following assumptions on the generator $g$ will be used in stating this corollary.

\begin{enumerate}
\renewcommand{\theenumi}{(H1'')}
\renewcommand{\labelenumi}{\theenumi}
\item\label{A:H1''} $\E\left[\int_0^\tau|g(s,0,0)|{\rm d}s\right]<+\infty.$
\renewcommand{\theenumi}{(H6')}
\renewcommand{\labelenumi}{\theenumi}
\item\label{A:H6'} There exists an $l\in(0,1)$ such that for each $y\in\R^k$ and $z \in\R^{k\times d}$,
    $$|g(t,y,z)-g(t,y,0)|\leq\gamma_t(g_t^1+g_t^2+|y|+|z|)^l, \ \ t\in[0,\tau],$$
    where $(\gamma_t)_{t\in[0,\tau]}$, $(g_t^1)_{t\in[0,\tau]}$ and $(g_t^2)_{t\in[0,\tau]}$ are three $(\F_t)$-progressively measurable nonnegative processes satisfying
    $$\int_0^\tau \left(\gamma_s^{\frac{1}{1-l}}+\gamma_s\right){\rm d}s\leq M,\ \ \E\left[\int_0^\tau g_s^1{\rm d}s\right]<+\infty \ \ {\rm and} \ \ \E\left[\sup\limits_{t\in[0,\tau]}|g_t^2|\right]<+\infty.$$
\end{enumerate}

\begin{cor}\label{cor:4.4}
Assume that the generator $g$ satisfies assumptions \ref{A:H1''}, \ref{A:H2}, \ref{A:H3c}, \ref{A:H4}, \ref{A:H5} and \ref{A:H6'}, and $\int_0^\tau\left(\mu_t+\nu_t^2\right){\rm d}t\leq M$. Then for each $\xi\in L_\tau^1(0;\R^k)$, BSDE (1.1) has a unique weighted $L^1$ solution $(y_t,z_t)_{t\in[0,\tau]}$ in the space of $\bigcap_{\theta\in(0,1)}\left[S_\tau^\theta(0;\R^k)\times M_\tau^\theta(0;\R^{k\times d})\right]$.
\end{cor}


\begin{rmk}
Assumption \ref{A:H6'} improves the corresponding assumptions used in \cite{Briand2003SPA} and \cite{Xiao2015}. Note that \ref{A:H3c} is strictly weaker than \ref{A:H3}. \cref{cor:4.4} strengthens both Theorem 6.3 of \cite{Briand2003SPA} and Theorem 4.1 of \cite{Xiao2015}.
\end{rmk}

\subsection{Examples}
In this section, we give several examples to illustrate that \cref{thm:4.1} can be applied to, but any existing results, such as Theorem 6.3 in \cite{Briand2003SPA} and Theorem 4.1 in \cite{Xiao2015}, cannot be applied to.

\begin{ex}\label{ex:111}
Let $\tau$ be a bounded stopping time, i.e., $\tau\leq T$ for a real $T>0$, and for each $y:=(y_1,\cdots,y_k)\in \R^k$ and $z\in \R^{k\times d}$, let
$$g(t,y,z):=(g^1(t,y,z),\cdots,g^k(t,y,z)), \ \ t\in[0,\tau],$$
where for each $i=1,\cdots,k,$
$$g^i(t,y,z):=e^{-|B_t|^3y_i}+|z|\wedge|z|^{\frac{2}{3}}.$$
By \cref{ex:1} in Section 3, we can deduce that the above generator $g$ satisfies \ref{A:H1''}, \ref{A:H2}, \ref{A:H3c}, \ref{A:H4}, \ref{A:H5} and \ref{A:H6'} with $|g(\cdot,0,0)|\equiv \sqrt{k}$, $\mu_\cdot\equiv 0$, $\nu_\cdot\equiv 1$, $l=\frac{2}{3}$, $\gamma_\cdot\equiv1$ and $g_\cdot^1=g_\cdot^2\equiv0$, but this $g$ does not satisfy assumption \ref{A:H3c} with $\alpha_\cdot\equiv1$ used in \cite{Briand2003SPA} and \cite{Xiao2015}. It follows from \cref{cor:4.4} that for each $\xi\in L_\tau^1(0;\R^k)$, BSDE \eqref{BSDE1.1} admits a unique weighted $L^1$ solution in $\bigcap_{\theta\in(0,1)}\left[S_\tau^\theta(0;\R^k)\times M_\tau^\theta(0;\R^{k\times d})\right]$.
\end{ex}

\begin{ex}\label{ex:5.6}
Let $k=1$, $\tau$ be a finite stopping time, and for each $(y,z)\in\R\times\R^{d}$,
$$g(t,y,z):=|B_t|^3(1-e^{y^3})+|B_t|^3\sin y+
\left(e^{-t}|z|\right)
\wedge\left(e^{-\beta\int_0^t|B_r|^3{\rm d}r}|B_s|^2|z|^{\frac{1}{3}}\right), \ \ t\in[0,\tau].$$
We can check that this generator $g$ satisfies assumptions \ref{A:H1'}, \ref{A:H2}, \ref{A:H4}-\ref{A:H6} with
$$g(t,0,0)=g_t^1=g_t^2\equiv0, \ \mu_t=|B_t|^3, \ \nu_t=e^{-t}, \ \gamma_t=e^{-\beta\int_0^t|B_r|^3{\rm d}r}|B_t|^2, \ \ {\rm and} \ \ l=\frac{1}{3}.$$
Moreover, it can also be verified that $g$ satisfies assumptions \ref{A:H3b} with $\tilde{\mu}_t=(|B_t|+1)^3$ and $\varphi(x)=e^{x^3}+2, \ x\geq0$. Then it follows from \cref{cor:h3s1} that for $\xi:=e^{-\beta\int_0^\tau |B_t|^2{\rm d}t}|B_\tau|\in L_\tau^1(\beta\mu_\cdot;\R^k)$, BSDE \eqref{BSDE1.1} admits a unique weighted $L^1$ solution in $\bigcap_{\theta\in(0,1)}\left[S_\tau^\theta(\beta\mu_\cdot;\R^k)\times M_\tau^\theta(\beta\mu_\cdot;\R^{k\times d})\right]$.
\end{ex}

\begin{ex}\label{ex:5.6}
Let $\tau$ be a stopping time taking values in $[0,+\infty)\cup \{+\infty\}$ and $\sigma$ be a finite stopping time. Define the following stopping time:
$$\overline{\sigma}:=\inf\left\{t\geq0:\int_0^t\left(e^{\frac{2}{3}\beta\int_0^s|B_r|^2{\bf 1}_{0\leq r\leq \sigma}{\rm d}r}|B_s|^\frac{4}{3}+|B_s|^2\right){\rm d}s\geq \frac{M}{2}\right\}\wedge \tau$$
with convention that $\inf \emptyset=+\infty$. Let
$$g(t,y,z):=(g^1(t,y,z),\cdots,g^k(t,y,z)), \ \ t\in[0,\tau],$$
where for each $i=1,\cdots,k,$
$$g^i(t,y,z):=-e^{|B_t|^4y_i}+|B_t|^2{\bf 1}_{0\leq t\leq \sigma}|y|+|B_t|{\bf 1}_{0\leq t\leq \overline{\sigma}}(
\sqrt{1+|z|}-1)+1.$$
It is straightforward to verify that this generator $g$ satisfies assumptions \ref{A:H1'}, \ref{A:H2}-\ref{A:H6} with
$$g(t,0,0)=g_t^1=g_t^2\equiv0, \ \mu_t=|B_t|^2{\bf 1}_{0\leq t\leq \sigma}, \ \nu_t=\gamma_t=|B_t|{\bf 1}_{0\leq t\leq \overline{\sigma}} \ \ {\rm and} \ \ l=\frac{1}{2}.$$
According to (iii) and (iv) of \cref{rmk-4} and \cref{thm:4.1}, we know that for each $\xi\in L_\tau^1(\beta\mu_\cdot;\R^k)$, BSDE \eqref{BSDE1.1} admits a unique weighted $L^1$ solution in $\bigcap_{\theta\in(0,1)}\left[S_\tau^\theta(\beta\mu_\cdot;\R^k)\times M_\tau^\theta(\beta\mu_\cdot;\R^{k\times d})\right]$.
\end{ex}

\subsection{Proof of Theorem \ref{thm:4.1}}
In this section, we will give the proof of Theorem 4.1. To begin with, we prove the uniqueness part, and then the existence part by two steps.

\begin{proof}[\bf Proof of the uniqueness part of Theorem \ref{thm:4.1}.]
Assume that $(y_t^i,z_t^i)_{t\in(0,\tau)}(i=1,2)$ are two weighted $L^1$ solutions of BSDE \eqref{BSDE1.1} such that $(e^{\beta\int_0^t \mu_r{\rm d}r}y_t^i)_{t\in(0,\tau)}(i=1,2)$ belong to the class $(D)$ and $(y_t^i,z_t^i)_{t\in(0,\tau)}(i=1,2)$ belong to $S_\tau^\theta(\beta\mu_\cdot;\R^k)\times M_\tau^\theta(\beta\mu_\cdot;\R^{k\times d})$ for some $\theta\in(l,1)$. For each integer $n>1$, let us introduce the following $(\F_t)$-stopping time:
$$\tau_n:=\inf\left\{t\geq0:\int_0^te^{2\beta\int_0^s\mu_r{\rm d}r}\left(|z_s^1|^2+|z_s^2|^2\right){\rm d}s\geq n\right\}\wedge \tau,$$
with convention that $\inf \emptyset=+\infty$.
Setting $\overline{y}_\cdot:=y_\cdot^1-y_\cdot^2$,~$\overline{z}_\cdot:=z_\cdot^1-z_\cdot^2$ and applying It\^{o}'s formula to $e^{\beta\int_0^t\mu_r{\rm d}r}|\overline{y}_t|$ yield that for each $t\geq0$,
\begin{align}\label{5.27}
\begin{split}
e^{\beta\int_{0}^{t\wedge\tau_n}\mu_s{\rm d}s}|\overline{y}_{t\wedge\tau_n}|
\leq& e^{\beta\int_{0}^{\tau_n}\mu_s{\rm d}s}|\overline{y}_{\tau_n}|-\int_{t\wedge\tau_n}^{\tau_n} e^{\beta\int_{0}^{s}\mu_r{\rm d}r}\left\langle \frac{\overline{y}_s}{|\overline{y}_s|}{\bf 1}_{|\overline{y}_s|\neq0},\overline{z}_s{\rm d}B_s\right\rangle\\
&+\int_{t\wedge\tau_n}^{\tau_n} e^{\beta\int_{0}^{s}\mu_r{\rm d}r}\left(\left\langle\frac{\overline{y}_s}{|\overline{y}_s|}{\bf 1}_{|\overline{y}_s|\neq0},g(s,y_s^1,z_s^1)-g(s,y_s^2,z_s^2)\right\rangle-\beta\mu_s|\overline{y}_s|\right){\rm d}s.
\end{split}
\end{align}
According to assumptions \ref{A:H4} and \ref{A:H6} for the generator $g$, for each $y_\cdot^1,y_\cdot^2\in\R^k$ and $z_\cdot^1,z_\cdot^2\in\R^{k\times d}$, we have
\begin{align*}
\begin{split}
&\left\langle\frac{\overline{y}_s}{|\overline{y}_s|}{\bf 1}_{|\overline{y}_s|\neq0},g(s,y_s^1,z_s^1)-g(s,y_s^2,z_s^2)\right\rangle-\beta\mu_s|\overline{y}_s|\\
&\ \  \leq \bigg|g(s,y_s^2,z_s^1)-g(s,y_s^2,0)+g(s,y_s^2,0)-g(s,y_s^2,z_s^2)\bigg|\\
&\ \  \leq 2\gamma_s\left(g_s^1+g_s^2+|y_s^2|+|z_s^1|+|z_s^2|\right)^l, \ \ s\in[0,\tau].
\end{split}
\end{align*}
In light of the last inequality along with the definition of $\tau_n$, by taking the conditional mathematical expectation with respect to $\F_{t\wedge\tau_m}$ on both sides of \eqref{5.27} we deduce that for each $t\geq0$ and $n\geq m\geq1$,
\begin{align}\label{5.29}
\begin{split}
&\E\left[e^{\beta\int_{0}^{t\wedge\tau_n}\mu_s{\rm d}s}|\overline{y}_{t\wedge\tau_n}|\bigg|\F_{t\wedge\tau_m}\right]\\
&\ \  \leq \E\left[e^{\beta\int_{0}^{\tau_n}\mu_s{\rm d}s}|\overline{y}_{\tau_n}|+2\int_{0}^{\tau} e^{\beta\int_{0}^{s}\mu_r{\rm d}r}\gamma_s\left(g_s^1+g_s^2+|y_s^2|+|z_s^1|+|z_s^2|\right)^l{\rm d}s\bigg|\F_{t\wedge\tau_m}\right].
\end{split}
\end{align}
Note that $e^{\beta\int_0^\cdot \mu_r{\rm d}r}\overline{y}_\cdot$ belongs to the class $(D)$. By letting $n\rightarrow+\infty$ on both sides of \eqref{5.29} we get that for each $t\geq0$ and $m\geq1$,
\begin{align}\label{5.30}
\begin{split}
\E\left[e^{\beta\int_{0}^{t\wedge\tau}\mu_s{\rm d}s}|\overline{y}_{t\wedge\tau}|\bigg|\F_{t\wedge\tau_m}\right]
\leq\E\left[2\int_{0}^{\tau} e^{\beta\int_{0}^{s}\mu_r{\rm d}r}\gamma_s\left(g_s^1+g_s^2+|y_s^2|+|z_s^1|+|z_s^2|\right)^l{\rm d}s\bigg|\F_{t\wedge\tau_m}\right].
\end{split}
\end{align}
Then by sending $m\rightarrow+\infty$ on both sides of \eqref{5.30} and using the martingale convergence theorem (see Corollary A.9 in Appendix C of \cite{Oksendal2005}) we obtain that for each $t\geq0$,
\begin{align*}
\begin{split}
e^{\beta\int_{0}^{t\wedge\tau}\mu_s{\rm d}s}|\overline{y}_{t\wedge\tau}|
\leq2\E\left[\int_{0}^{\tau} e^{\beta\int_{0}^{s}\mu_r{\rm d}r}\gamma_s\left(g_s^1+g_s^2+|y_s^2|+|z_s^1|+|z_s^2|\right)^l{\rm d}s\bigg|\F_{t\wedge\tau}\right].
\end{split}
\end{align*}
Moreover, by virtue of $\frac{\theta}{l}>1$, it follows from Doob's inequality and the last inequality that there exists a constant $C_l^\theta>0$ depending on $l$ and $\theta$ such that
\begin{align}\label{5.32}
\begin{split}
\E\left[\sup_{t\in[0,\tau]}\left(e^{\beta\int_{0}^{t}\mu_s{\rm d}s}|\overline{y}_{t}|\right)^{\frac{\theta}{l}}\right]
\leq C_l^\theta\E\left[\left(\int_{0}^{\tau} e^{\beta\int_{0}^{s}\mu_r{\rm d}r}\gamma_s\left(g_s^1+g_s^2+|y_s^2|+|z_s^1|+|z_s^2|\right)^l{\rm d}s\right)^{\frac{\theta}{l}}\right].
\end{split}
\end{align}
Next we will prove that $G:=\int_{0}^{\tau} e^{\beta\int_{0}^{s}\mu_r{\rm d}r}\gamma_s\left(g_s^1+g_s^2+|y_s^2|+|z_s^1|+|z_s^2|\right)^l{\rm d}s\in L_\tau^{\frac{\theta}{l}}(0;\R^k)$. In fact, by H\"{o}lder's inequality we have
\begin{align}\label{gs1}
\begin{split}
\int_0^\tau e^{\beta\int_{0}^{s}\mu_r{\rm d}r}\gamma_s(g_s^1)^l{\rm d}s\leq \left(\int_0^\tau e^{\beta\int_{0}^{s}\mu_r{\rm d}r}g_s^1{\rm d}s\right)^l\left(\int_0^\tau \left(e^{(1-l)\beta\int_0^s \mu_r{\rm d}r}\gamma_s\right)^{\frac{1}{1-l}}{\rm d}s\right)^{1-l},
\end{split}
\end{align}
\begin{align}\label{gs2}
\begin{split}
\int_0^\tau e^{\beta\int_{0}^{s}\mu_r{\rm d}r}\gamma_s(g_s^2)^l{\rm d}s\leq \sup_{s\in[0,\tau]}\left(e^{l\beta\int_{0}^{s}\mu_r{\rm d}r}|g_s^2|^l\right)\int_0^\tau e^{(1-l)\beta\int_{0}^{s}\mu_r{\rm d}r}\gamma_s{\rm d}s,
\end{split}
\end{align}
\begin{align}\label{ys}
\begin{split}
\int_0^\tau e^{\beta\int_{0}^{s}\mu_r{\rm d}r}\gamma_s|y_s^2|^l{\rm d}s\leq \sup_{s\in[0,\tau]}\left(e^{l\beta\int_{0}^{s}\mu_r{\rm d}r}|y_s^2|^l\right)\int_0^\tau e^{(1-l)\beta\int_{0}^{s}\mu_r{\rm d}r}\gamma_s{\rm d}s,
\end{split}
\end{align}
and
\begin{align}\label{zs}
\begin{split}
\int_0^\tau e^{\beta\int_{0}^{s}\mu_r{\rm d}r}\gamma_s|z_s^1|^l{\rm d}s\leq \left(\int_0^\tau e^{2\beta\int_{0}^{s}\mu_r{\rm d}r}|z_s^1|^2{\rm d}s\right)^{\frac{l}{2}}\left(\int_0^\tau \left(e^{(1-l)\beta\int_0^s \mu_r{\rm d}r}\gamma_s\right)^{\frac{2}{2-l}}{\rm d}s\right)^{\frac{2-l}{2}}.
\end{split}
\end{align}
Note that $\E\left[\int_0^\tau e^{\beta\int_0^s \mu_r{\rm d}r}g_s^1{\rm d}s\right]<+\infty$, $\E\left[\sup\limits_{t\in[0,\tau]}\left(e^{\beta\int_0^t \mu_r{\rm d}r}|g_t^2|\right)\right]<+\infty$, $y_\cdot^2$ belongs to $S_\tau^\theta(\beta\mu_\cdot;\R^k)$ and $z_\cdot^1,z_\cdot^2$ belong to $M_\tau^\theta(\beta\mu_\cdot;\R^{k\times d})$. It then follows from assumption \ref{A:H6}, (iii) of \cref{rmk-4} and \eqref{gs1}-\eqref{zs} that $G\in L_\tau^{\frac{\theta}{l}}(0;\R^k)$. In light of \eqref{5.32}, we can draw a conclusion that $\overline{y}_\cdot$ belongs to $S_\tau^{\frac{\theta}{l}}(\beta\mu_\cdot;\R^k)$.

Furthermore, $(\overline{y}_t,\overline{z}_t)$ is a solution of the following BSDE:
\begin{align*}
  \overline{y}_t=\int_t^\tau \overline{g}(s,\overline{y}_s,\overline{z}_s){\rm d}s-\int_t^\tau \overline{z}_s{\rm d}B_s, \ \ t\in[0,\tau],
\end{align*}
where $\overline{g}(t,y,z):=g(t,y+y_t^2,z+z_t^2)-g(t,y_t^2,z_t^2)$ for each $(y,z)\in\R^k\times\R^{k\times d}$. Using assumptions \ref{A:H4} and \ref{A:H5} for the generator $g$, we have
$$\left<\hat{y},\overline{g}(t,y,z)\right>\leq \mu_{t}|y|+\nu_t|z|, \ t\in[0,\tau],$$
which means that $\overline{g}$ satisfies assumption \ref{A:A} with $u_t:=\mu_t$, $v_t:=\nu_t$ and $f_t:=0$. Therefore, it follows from \cref{pro:1.2} with $p=\frac{\theta}{l}>1$ that $(\overline{y}_\cdot,\overline{z}_\cdot)\equiv(0,0)$. The uniqueness part is proved.
\end{proof}

\begin{proof}[\bf Proof of the existence part of Theorem \ref{thm:4.1}.]
The proof is divided into two steps. We first consider the case where the generator $g$ is independent of the variable $z$ under assumptions \ref{A:H1'}, \ref{A:H2}-\ref{A:H4}, and then the general case by using the Picard's iteration method.

{\bf First Step:} We firstly consider the case where the generator $g$ is independent of the variable $z$. Similar to \cite{Li2024}, without loss of generality, we can assume that the process $\alpha_\cdot$ in \ref{A:H3} satisfies
$$\alpha_t\leq e^{-\beta\int_0^t\mu_s{\rm d}s}\leq1, \ \ t\in[0,\tau].$$

For each $(x,y)\in \R^k\times \R^{k}$, $r>0$ and $n\geq1$, let $q_r(x):=\frac{xr}{|x|\vee r}$, $\xi_n:=q_{n\alpha_\tau}(\xi)$ and $$g_n(t,y):=g(t,y)-g(t,0)+q_{ne^{-t}\alpha_t}(g(t,0)), \ \ t\in[0,\tau].$$
Then for each $n\geq1$, we have
$$|\xi_n|\leq n\alpha_\tau, \ {\rm and} \ |g_n(t,0)|\leq ne^{-t}\alpha_t, \ \ t\in[0,\tau].$$
Thus $g_n(t,y)$ satisfies assumptions \ref{A:H1}-\ref{A:H4} with $p=2$. Note that $\xi_n\in L_\tau^2(\beta\mu_\cdot;\R^k)$ and $\int_0^\tau \nu_t^2{\rm d}t\leq M$. It follows from Theorem 3.1 in \cite{Li2024} that for each $n\geq1$ the following BSDE \eqref{BSDE-2} admits a unique $L^2$ solution $\{(y_t^n,z_t^n)\}_{t\in[0,\tau]}$ in $S_\tau^2(\beta\mu_\cdot;\R^k)\times M_\tau^2(\beta\mu_\cdot;\R^{k\times d})$:
\begin{align}\label{BSDE-2}
  y_t^n=\xi_n+\int_t^\tau g_n(s,y_s^n){\rm d}s-\int_t^\tau z_s^n{\rm d}B_s, \ \ t\in[0,\tau].
\end{align}
In the sequel, for each pair of integers $n,i\geq1$, let
$$
\hat{\xi}^{n,i}:=\xi_{n+i}-\xi_{n}, \ \ \hat{y}_\cdot^{n,i}:=y_\cdot^{n+i}-y_\cdot^{n}, \ \ \hat{z}_\cdot^{n,i}:=z_\cdot^{n+i}-z_\cdot^{n}.
$$
Applying It\^{o}'s formula to $e^{\beta\int_{0}^{t}\mu_s{\rm d}s}|\hat{y}_{t}^{n,i}|$ yields that for each $t\geq0$,
\begin{align}\label{5.34}
\begin{split}
e^{\beta\int_{0}^{t\wedge\tau}\mu_s{\rm d}s}|\hat{y}_{t\wedge\tau}^{n,i}|&\leq \int_{t\wedge\tau}^{\tau} e^{\beta\int_{0}^{s}\mu_r{\rm d}r}\left(\left\langle\frac{\hat{y}_s^{n,i}}{|\hat{y}_s^{n,i}|}{\bf 1}_{|\hat{y}_s^{n,i}|\neq0},g_{n+i}(s,y_s^{n+i})-g_{n}(s,y_s^{n})
\right\rangle-\beta\mu_s|\hat{y}_s^{n,i}|\right){\rm d}s\\
&\ \ \ \ +e^{\beta\int_{0}^{\tau}\mu_s{\rm d}s}|\hat{\xi}^{n,i}|-\int_{t\wedge\tau}^{\tau} e^{\beta\int_{0}^{s}\mu_r{\rm d}r}\left\langle \frac{\hat{y}_s^{n,i}}{|\hat{y}_s^{n,i}|}{\bf 1}_{|\hat{y}_s^{n,i}|\neq0},\hat{z}_s^{n,i}{\rm d}B_s\right\rangle.
\end{split}
\end{align}
Using assumption \ref{A:H4} on $g$ and the definition of $g_n$, we have
\begin{align*}
\begin{split}
&e^{\beta\int_{0}^{s}\mu_r{\rm d}r}\left(\left\langle\frac{\hat{y}_s^{n,i}}{|\hat{y}_s^{n,i}|}{\bf 1}_{|\hat{y}_s^{n,i}|\neq0},g_{n+i}(s,y_s^{n+i})-g_{n}(s,y_s^{n})
\right\rangle-\beta\mu_s|\hat{y}_s^{n,i}|\right)\\
&\ \  \leq  e^{\beta\int_{0}^{s}\mu_r{\rm d}r}\bigg|g_{n+i}(s,y_s^{n})-g_{n}(s,y_s^{n})\bigg| \leq  e^{\beta\int_{0}^{s}\mu_r{\rm d}r}|g(s,0)|{\bf 1}_{|g(s,0)|>ne^{-s}\alpha_s}, \ \ s\in[0,\tau].
\end{split}
\end{align*}
On the other hand, it is easy to verify that
$$
\left\{\int_{0}^{t}e^{\beta\int_{0}^{s}\mu_r{\rm d}r}\left\langle \frac{\hat{y}_s^{n,i}}{|\hat{y}_s^{n,i}|}{\bf 1}_{|\hat{y}_s^{n,i}|\neq0},\hat{z}_s^{n,i}{\rm d}B_s\right\rangle\right\}_{t\in[0,\tau]}
$$
is a uniformly integrable martingale. In fact, by the BDG inequality and the fact of $z_\cdot^n\in M_\tau^2(\beta\mu_\cdot;\R^{k\times d})$ we have
\begin{align*}
\begin{split}
&\E\left[\sup_{t\in[0,\tau]}\bigg|\int_{0}^{t}e^{\beta\int_{0}^{s}\mu_r{\rm d}r}\left\langle \frac{\hat{y}_s^{n,i}}{|\hat{y}_s^{n,i}|}{\bf 1}_{|\hat{y}_s^{n,i}|\neq0},\hat{z}_s^{n,i}{\rm d}B_s\right\rangle\bigg|\right]\\
&\ \  \leq 2\sqrt{2}\E\left[\left(\int_{0}^{\tau}e^{2\beta\int_{0}^{s}\mu_r{\rm d}r}|\hat{z}_s^{n,i}|^2 {\rm d}{s}\right)^{\frac{1}{2}}\right]<+\infty.
\end{split}
\end{align*}
Then, by taking the conditional mathematical expectation with respect to $\F_{t\wedge\tau}$ on both sides of \eqref{5.34} and using the last two inequalities, we get that for each $t\geq0$,
\begin{align}\label{5*3.4}
\begin{split}
e^{\beta\int_{0}^{t\wedge\tau}\mu_s{\rm d}s}|\hat{y}_{t\wedge\tau}^{n,i}|
\leq\E\left[e^{\beta\int_{0}^{\tau}\mu_s{\rm d}s}|\xi|{\bf 1}_{|\xi|>n\alpha_\tau}+\int_{0}^{\tau}e^{\beta\int_{0}^{s}\mu_r{\rm d}r}|g(s,0)|{\bf 1}_{|g(s,0)|>ne^{-s}\alpha_s}{\rm d}s\bigg|\F_{t\wedge\tau}\right]
\end{split}
\end{align}
and then
\begin{align}\label{5*3.4.}
\begin{split}
\E\left[e^{\beta\int_{0}^{t\wedge\tau}\mu_s{\rm d}s}|\hat{y}_{t\wedge\tau}^{n,i}|\right]\leq \E\left[e^{\beta\int_{0}^{\tau}\mu_s{\rm d}s}|\xi|{\bf 1}_{|\xi|>n\alpha_\tau}+\int_{0}^{\tau}e^{\beta\int_{0}^{s}\mu_r{\rm d}r}|g(s,0)|{\bf 1}_{|g(s,0)|>ne^{-s}\alpha_s}{\rm d}s\right].
\end{split}
\end{align}
By virtue of Lemma 6.1 in \cite{Briand2003SPA} and \eqref{5*3.4}, we know that for any $\theta\in(0,1)$,
\begin{align}\label{5*34}
\begin{split}
&\E\left[\sup\limits_{t\in[0,\tau]}\left(e^{\beta\int_{0}^{t\wedge\tau}\mu_s{\rm d}s}|\hat{y}_{t\wedge\tau}^{n,i}|\right)^\theta\right]\\
&\ \  \leq
\frac{1}{1-\theta}\left(\E\left[e^{\beta\int_{0}^{\tau}\mu_s{\rm d}s}|\xi|{\bf 1}_{|\xi|>n\alpha_\tau}+\int_{0}^{\tau}e^{\beta\int_{0}^{s}\mu_r{\rm d}r}|g(s,0)|{\bf 1}_{|g(s,0)|>ne^{-s}\alpha_s}{\rm d}s\right]\right)^\theta.
\end{split}
\end{align}
Therefore, according to $\xi\in L_\tau^1(\beta\mu_\cdot;\R^k)$, assumption \ref{A:H1'} and \eqref{5*34}, we deduce that $\{y_\cdot^n\}^{+\infty}_{n=1}$ is a Cauchy sequence in $S_\tau^\theta(\beta\mu_\cdot;\R^{k})$. Let $(y_t)_{t\in [0,\tau]}$ denote the limit of this sequence. Then $(e^{\beta\int_{0}^{t}\mu_s{\rm d}s}y_t)_{t\in [0,\tau]}$ belongs to the class (D) (in light of \eqref{5*3.4.}) and $y_\cdot \in S_\tau^\theta(\beta\mu_\cdot;\R^{k})$ for any $\theta\in(0,1)$.

Furthermore, we show that the sequence $\{z_\cdot^n\}^{+\infty}_{n=1}$ is a Cauchy sequence in $M_\tau^\theta(\beta\mu_\cdot;\R^{k\times d})$ for any $\theta\in(0,1)$. It is clear that $(\hat{y}_\cdot^{n,i},\hat{z}_\cdot^{n,i})$ is a pair of solutions of the following BSDE:
\begin{align*}
  \hat{y}_t^{n,i}=\hat{\xi}^{n,i}+\int_t^\tau \hat{g}(s,\hat{y}_s^{n,i}){\rm d}s-\int_t^\tau \hat{z}_s^{n,i}{\rm d}B_s, \ \ t\in[0,\tau],
\end{align*}
where the generator $\hat{g}(t,y):=g_{n+i}(t,y+y_t^n)-g_{n}(t,y_t^n)$ for each $y\in \R^k$. According to assumption \ref{A:H4} on the generator $g_n$ we have
$$\left<\hat{y},\hat{g}(t,y)\right>\leq \mu_{t}|y|+|g(t,0)|{\bf 1}_{|g(t,0)|>n e^{-t}\alpha_t}, \ \ t\in[0,\tau],$$
which means that $\hat{g}$ satisfies assumption \ref{A:A} with $u_t:=\mu_t$, $v_t:=0$ and $f_t:=|g(t,0)|{\bf 1}_{|g(t,0)|>n e^{-t}\alpha_t}$. Then it follows from \eqref{pro:1.1-2} of \cref{pro:1.1} with $p=\theta$ and $t=0$ that for each $\theta\in(0,1)$, there exists a constant $C_{\theta,\rho,M}>0$ depending only on $\theta$, $\rho$ and $M$ such that
\begin{align*}
\begin{split}
&\E\left[\left(\int_{0}^{\tau}e^{2\beta\int_{0}^{s}\mu_r{\rm d}r}|\hat{z}_s^{n,i}|^2{\rm d}s\right)^{\frac{\theta}{2}}\right]\\
&\ \  \leq
C_{\theta,\rho,M}\left(\E\left[\sup_{s\in[0,\tau]}
\left(e^{\beta{\int_{0}^{s}\mu_r{\rm d}r}}|\hat{y}_s^{n,i}|\right)^\theta\right]+ \E\left[\left(\int_{0}^{\tau}e^{\beta{\int_{0}^{s}\mu_r{\rm d}r}}|g(s,0)|{\bf 1}_{|g(s,0)|>n e^{-s}\alpha_s}{\rm d}s\right)^\theta\right]\right).
\end{split}
\end{align*}
Thus for any $\theta\in(0,1)$, the sequence $\{z_\cdot^n\}^{+\infty}_{n=1}$ is a Cauchy sequence in $M_\tau^\theta(\beta\mu_\cdot;\R^{k\times d})$. Let $(z_t)_{t\in[0,\tau]}$ denote the limit of $\{z_\cdot^n\}^{+\infty}_{n=1}$. Then, $z_\cdot\in M_\tau^\theta(\beta\mu_\cdot;\R^{k\times d})$. Finally, by taking the limit in the sense of ucp for BSDE \eqref{BSDE-2} we obtain that $(y_t,z_t)_{t\in[0,\tau]}$ is the desired $L^1$ solution of BSDE \eqref{BSDE1.1} in $\bigcap_{\theta\in(0,1)}[S_\tau^\theta(\beta\mu_\cdot;\R^k)$ $\times M_\tau^\theta(\beta\mu_\cdot;\R^{k\times d})]$ such that $(e^{\beta\int_{0}^{t}\mu_s{\rm d}s}y_t)_{t\in [0,\tau]}$ belongs to the class (D).

{\bf Second Step:} We use Picard's iterative procedure to prove the existence of Theorem \ref{thm:4.1} under the general case. Let $(y_\cdot^0,z_\cdot^0):=(0,0)$. Assume that $g$ satisfies assumptions \ref{A:H1'}, \ref{A:H2}-\ref{A:H6}. Recalling that $\int_0^\tau \nu_t^2{\rm d}t\leq M$ and $\int_0^\tau \mu_t{\rm d}t<+\infty$. Based on the result of the first step, we can define the process sequence $\{(y_t^{n+1},z_t^{n+1})_{t\in[0,\tau]}\}_{n=1}^\infty$ recursively for each $n\geq0$,
\begin{align}\label{BSDE3.1}
  y_t^{n+1}=\xi+\int_t^\tau g(s,y_s^{n+1},z_s^n){\rm d}s-\int_t^\tau z_s^{n+1}{\rm d}B_s, \ \ t\in[0,\tau].
\end{align}
In fact, we will first prove that for each $\xi\in L_\tau^\theta(\beta\mu_\cdot;\R^k)$ and $z_\cdot\in \bigcap_{\theta\in(0,1)}M_\tau^\theta(\beta\mu_\cdot;\R^{k\times d})$, the generator $g(t,y,z_t)$ satisfies assumptions \ref{A:H1'} and \ref{A:H2}-\ref{A:H4}. Firstly, it is obvious that $g(t,y,z_t)$ satisfies \ref{A:H2} and \ref{A:H4}. Furthermore, by \ref{A:H6} of $g$ we derive that
\begin{align*}
\E\left[\int_0^\tau e^{\beta \int_0^s{\mu}_r{\rm d}r}|g(s,0,z_s)|{\rm d}s\right]\leq \E\left[\int_0^\tau e^{\beta \int_0^s{\mu}_r{\rm d}r}|g(s,0,0)|{\rm d}s\right]+\E\left[e^{\beta \int_0^s{\mu}_r{\rm d}r}\gamma_s(g_s^1+g_s^2+|z_s|)^l{\rm d}s\right].
\end{align*}
From assumption \ref{A:H1} of $g$, \eqref{gs1}, \eqref{gs2} and \eqref{zs} along with the last inequality, we know that \ref{A:H1'} is also true for $g(t,y,z_t)$. Note that \ref{A:H3} also holds for $g$ when $\alpha_\cdot$ is replaced with
\begin{align}\label{alphat}
\hat{\alpha}_t:={\alpha}_t\wedge e^{-\beta \int_0^t{\mu}_r{\rm d}r-t}, \ \ t\in[0,\tau].
\end{align}
We deduce that for each $r\in \R_+$ and $t\in[0,\tau]$,
\begin{align*}
\begin{split}
\overline{\psi}_{r}^{\hat{\alpha}_\cdot}(t):&=\sup_{|y|\leq r\hat{\alpha}_t}\left\{ \left|g(t,y,z_t)-g(t,0,z_t)\right|\right\}\\
&=\sup_{|y|\leq r\hat{\alpha}_t}\left\{ \left|g(t,y,z_t)-g(t,y,0)+g(t,y,0)-g(t,0,0)+g(t,0,0)-g(t,0,z_t)\right|\right\}\\
&\leq 2\gamma_t(g_t^1+g_t^2+r\hat{\alpha}_t+|z_t|)^l
+{\psi}_{r}^{\hat{\alpha}_\cdot}(t),
\end{split}
\end{align*}
which along with \eqref{gs1}, \eqref{gs2}, \eqref{zs}, \eqref{alphat} and \ref{A:H6} implies that assumption \ref{A:H3} is true for $g(t,y,z_t)$. Thus, based on the result of the first step, for each $n\geq0$, BSDE \eqref{BSDE3.1} has a unique weighted $L^1$ solution $(y_t^{n+1},z_t^{n+1})_{t\in[0,\tau]}$ in the space of $\bigcap_{\theta\in(0,1)}H_\tau^\theta(\beta\mu_\cdot;\R^{k}\times\R^{k\times d})$ such that $(e^{\beta\int_{0}^{t}\mu_s{\rm d}s}y_t^{n+1})_{t\in [0,\tau]}$ belongs to the class (D).

Similar to the proof of the uniqueness part of \cref{thm:4.1}, by virtue of the assumptions \ref{A:H4} and \ref{A:H6} of $g$, we can deduce that for each $m\geq n\geq1$,
\begin{align*}
\begin{split}
e^{\beta\int_{0}^{t\wedge\tau}\mu_s{\rm d}s}|y_{t\wedge\tau}^{m}-y_{t\wedge\tau}^{n}|
\leq2\E\left[I_{m,n}\bigg|\F_{t\wedge\tau}\right], \ \ t\geq0,
\end{split}
\end{align*}
where $I_{m,n}:=\int_{0}^{\tau} e^{\beta\int_{0}^{s}\mu_r{\rm d}r}\gamma_s\left(g_s^1+g_s^2+|y_s^n|+|z_s^{m-1}|+|z_s^{n-1}|\right)^l{\rm d}s$, and that there exists a constant $q>1$ such that $I_{m,n}\in L_\tau^{q}(0;\R^k)$. Combining the last inequality and Doob's inequality we deduce that there exists a constant $c_q>0$ only depending on this $q$ such that for each $m\geq n\geq1$,
$$\E\left[\sup\limits_{t\in[0,\tau]}\left(e^{\beta\int_{0}^{t\wedge\tau}\mu_s{\rm d}s}|y_{t}^{m}-y_{t}^{n}|\right)^q\right]\leq c_q\E[I_{m,n}^q]< +\infty.$$
Thus, $(y_{t}^{m}-y_{t}^{n})_{t\in [0,\tau]}$ belongs to $S_\tau^q(\beta\mu_\cdot;\R^{k})$ for some $q>1$.

For each $n\geq1$, let $\hat{y}_\cdot^n:=y_\cdot^{n+1}-y_\cdot^{n}$ and $\hat{z}_\cdot^n:=z_\cdot^{n+1}-z_\cdot^{n}$. Then $(\hat{y}_\cdot^n,\hat{z}_\cdot^n)$ solves the following BSDE:
\begin{align}\label{bsdenn}
  \hat{y}_t^n=\int_t^\tau g^n(s,\hat{y}_s^n){\rm d}s-\int_t^\tau \hat{z}_s^n {\rm d}B_s, \ \ t\in[0,\tau],
\end{align}
where $g^n(t,y):=g(t,y+y_t^n,z_t^n)-g(t,y_t^n,z_t^{n-1})$ for each $y\in \R^k$. In light of assumptions \ref{A:H4} and \ref{A:H6} on the generator $g$, we have
$$\left<\hat{y},g^n(t,y)\right>\leq \mu_{t}|y|+2(g_t^1+g_t^2+|y_t^n|+|z_t^{n-1}|+|z_t^n|)^l, \ \ t\in[0,\tau],$$
which implies that assumption \ref{A:A} comes true for $g^n$ with $u_t:=\mu_{t}$, $v_t:=0$ and $f_t:=2(g_t^1+g_t^2+|y_t^n|+|z_t^{n-1}|+|z_t^n|)^l$. Thus, by virtue of $I_{n+1,n}\in L_\tau^{q}(0;\R^k)$, it follows from \cref{pro:1.2} with $p=q$ and $r=t$ that $(\hat{z}_t^n)_{t\in[0,\tau]}$ belongs to $M_\tau^q(\beta\mu_\cdot;\R^{k\times d})$. On the other hand, by assumptions \ref{A:H4} and \ref{A:H5} on the generator $g$ we have
$$\left<\hat{y},g^n(t,y)\right>\leq \mu_{t}|y|+\nu_t\hat{z}_t^{n-1}, \ \ t\in[0,\tau],$$
which means that $g^n$ satisfies assumption \ref{A:A} with $u_t:=\mu_{t}$, $v_t:=0$ and $f_t:=\nu_t\hat{z}_t^{n-1}$. Thus using \cref{pro:1.2} with $p=q$ and $r=t$ yields that there exists a constant $C_{q,\rho}>0$ depending only on $q$ and $\rho$ such that for each $n\geq 1$ and $t\geq 0$,
\begin{align}\label{555.43}
\begin{split}
&\E\left[\sup\limits_{s\in[t\wedge\tau,\tau]}\left(e^{q\beta\int_{0}^{s}\mu_s{\rm d}s}|\hat{y}_s^n|^q\right)+\left(\int_{t\wedge\tau}^\tau e^{2\beta\int_{0}^{s}\mu_r{\rm d}r}|\hat{z}_s^n|^2{\rm d}s\right)^{\frac{q}{2}}\bigg|\F_{t\wedge\tau}\right]\\
&\ \  \leq
C_{q,\rho}\E\left[\left(\int_{t\wedge\tau}^\tau e^{\beta\int_{0}^{s}\mu_r{\rm d}r}|\nu_s||\hat{z}_s^{n-1}|{\rm d}s\right)^q\bigg|\F_{t\wedge\tau}\right]\\
&\ \  \leq  C_{q,\rho}\E\left[\left(\int_{t\wedge\tau}^\tau \nu_s^2{\rm d}s\right)^{\frac{q}{2}}\bigg|\F_{t\wedge\tau}\right]\E\left[\left(\int_{t\wedge\tau}^\tau e^{2\beta\int_{0}^{s}\mu_r{\rm d}r}|\hat{z}_s^{n-1}|^2{\rm d}s\right)^{\frac{q}{2}}\bigg|\F_{t\wedge\tau}\right].
\end{split}
\end{align}
In the sequel, let us choose a sufficiently large integer $N$ such that $$\frac{M^{\frac{q}{2}}}{N}\leq \frac{1}{2C_{q,\rho}}.$$
And, we divide the time interval $[0,\tau]$ into  a finite number of small intervals $[\tau_{j-1},\tau_{j}], j=1,2,...,N$ by the following $(\F_t)$-stopping times:
\begin{align*}
\begin{split}
\tau_0&:=0,\\
\tau_1&:=\inf \left\{s\geq0: \left(\int_{0}^{s}\nu_{r}^{2}{\rm d}r\right)^{\frac{q}{2}} \geq \frac{M^{\frac{q}{2}}}{N}\right\} \wedge \tau,\\
&\vdots\\
\tau_j&:=\inf \left\{s\geq \tau_{j-1}: \left(\int_{0}^{s}\nu_{r}^{2}{\rm d}r\right)^{\frac{q}{2}} \geq \frac{jM^{\frac{q}{2}}}{N}\right\} \wedge \tau,\\
&\vdots\\
\tau_N&:=\inf \left\{s\geq \tau_{N-1}: \left(\int_{0}^{s}\nu_{r}^{2}{\rm d}r\right)^{\frac{q}{2}} \geq \frac{NM^{\frac{q}{2}}}{N}\right\} \wedge \tau=\tau,
\end{split}
\end{align*}
with convention that $\inf \emptyset=+\infty$.
Thus for any $[\tau_{j-1},\tau_{j}]\subset[0,\tau], j=1,2,...,N,$ we have
\begin{align}\label{*}
\begin{split}
\left(\int_{\tau_{j-1}}^{\tau_{j}} \nu_s^2{\rm d}s\right)^{\frac{q}{2}}\leq \frac{1}{2C_{q,\rho}}.
\end{split}
\end{align}
From the proof of \cref{pro:1.2}, it is not very difficult to observe that the time variable $t$ appearing in \eqref{555.43} can be replaced with any $(\F_t)$-stopping time. Then we have
\begin{align*}
\begin{split}
&\E\left[\sup\limits_{s\in[0,\tau]}\left(e^{q\beta\int_{0}^{s}\mu_r{\rm d}r}|\hat{y}_{s}^n{\bf 1}_{\tau_{N-1}\leq s\leq\tau}|^q\right)+\left(\int_0^\tau e^{2\beta\int_{0}^{s}\mu_r{\rm d}r}|\hat{z}_s^n{\bf 1}_{\tau_{N-1}\leq s\leq\tau}|^2{\rm d}s\right)^{\frac{q}{2}}\right]\\
&\ \  \leq  \frac{1}{2}\E\left[\left(\int_0^\tau e^{2\beta\int_{0}^{s}\mu_r{\rm d}r}|\hat{z}_s^{n-1}{\bf 1}_{\tau_{N-1}\leq s\leq\tau}|^2{\rm d}s\right)^{\frac{q}{2}}\right].
\end{split}
\end{align*}
By induction we get that for each $n\geq1$,
\begin{align*}
\begin{split}
\|(\hat{y}_t^n{\bf 1}_{\tau_{N-1}\leq t\leq\tau},\hat{z}_t^n{\bf 1}_{\tau_{N-1}\leq t\leq\tau})\|^q_{p;\beta\mu_\cdot}\leq \left(\frac{1}{2}\right)^{n-1}\|(\hat{y}_t^1{\bf 1}_{\tau_{N-1}\leq t\leq\tau},\hat{z}_t^1{\bf 1}_{\tau_{N-1}\leq t\leq\tau})\|^q_{p;\beta\mu_\cdot}.
\end{split}
\end{align*}
Since $\|(\hat{y}_t^1,\hat{z}_t^1)\|^q_{\beta\mu_\cdot;q}<+\infty$, it follows immediately that $\{(y_t^n-y_t^1,z_t^n-z_t^1)_{t\in[\tau_{N-1},\tau]}\}_{n=1}^\infty$ converges to some $(Y_\cdot,Z_\cdot)$ on $[\tau_{N-1},\tau]$ in the space of $S_\tau^q(\beta\mu_\cdot;\R^{k})\times M_\tau^q(\beta\mu_\cdot;\R^{k\times d})$. Then $\{(y_t^n,z_t^n)_{t\in[\tau_{N-1},\tau]}\}_{n=1}^\infty$ converges to $(y_t:=Y_t+y_t^1,z_t:=Z_t+z_t^1)_{t\in[\tau_{N-1},\tau]}$
in $\bigcap_{\theta\in(0,1)}S_\tau^\theta(\beta\mu_\cdot;\R^{k})\times M_\tau^\theta(\beta\mu_\cdot;\R^{k\times d})$ such that $(e^{\beta\int_{0}^{t}\mu_s{\rm d}s}y_t)_{t\in[\tau_{N-1},\tau]}$ belongs to the class (D). Thus BSDE $(\xi,\tau,g)$ admits a unique weighted $L^1$ solution $(y_\cdot,z_\cdot)$ on $[\tau_{N-1},\tau]$. Furthermore, noticing that \eqref{*} also holds when $j=N-1$, respectively replacing $\xi$, $\tau_N$, $\tau_{N-1}$ by $y_{\tau_{N-1}}$, $\tau_{N-1}$, $\tau_{N-2}$ yields that BSDE \eqref{BSDE1.1} admits a unique weighted $L^1$ solution on $[\tau_{N-1},\tau_{N-2}]$. By repeating the above proof process on $[\tau_{N-2},\tau_{N-3}],\cdots,[0,\tau_{1}]$ for finite times, we deduce that BSDE \eqref{BSDE1.1} admits a weighted $L^1$ solution $(y_t,z_t)_{t\in[0,\tau]}\in\bigcap_{\theta\in(0,1)}\left[S_\tau^\theta(\beta\mu_\cdot;\R^k)\times M_\tau^\theta(\beta\mu_\cdot;\R^{k\times d})\right]$ such that $(e^{\beta\int_{0}^{t}\mu_s{\rm d}s}y_t)_{t\in[0,\tau]}$ belongs to the class (D). The existence part of \cref{thm:4.1} is proved.
\end{proof}

\begin{rmk}\label{rmk4.5}
In order to prove \cref{thm:4.1}, we systematically employ some methods used in \cite{Briand2003SPA}, \cite{Xiao2015}, \cite{Liu2020}, \cite{LiT2019} and \cite{Li2024}, and develop some innovative ideas. In particular, when we verify that the generator of BSDE \eqref{BSDE3.1} satisfies the conditions in the first step, the appearance of $\hat{\alpha}_\cdot$ in \ref{A:H3} plays a key role, see \eqref{alphat} for more details. We also would like to mention that it is interesting and unsolved whether the additional condition of $\int_0^t \nu_s^2 {\rm d}s\leq M$ can be eliminated in \cref{thm:4.1}.
\end{rmk}

\section{Comparison theorems}
\setcounter{equation}{0}
In this section, we let $k=1$, and establish two important comparison theorems for the weighted $L^p~(p>1)$ solutions and the weighted $L^1$ solutions of one-dimensional BSDEs, respectively.

\begin{thm}\label{thm:com1}
Assume that $\xi$ and $\xi'$ are two terminal values, $g$ and $g^{\prime}$ are two generators, and $(Y_\cdot, Z_\cdot)$ and $\left(Y_\cdot^{\prime}, Z_\cdot^{\prime}\right)$ are, respectively, a solution of BSDE $(\xi,\tau,g)$ and BSDE $\left(\xi^{\prime}, \tau, g^{\prime}\right)$. If $\xi \leq \xi^{\prime}$, $(Y_\cdot-Y_\cdot^{\prime})^+\in S_\tau^p(a_\cdot;\R^k)$ for some $p>1$ and either of the following two conditions is satisfied:\vspace{0.2cm}

(i) $g$ satisfies \ref{A:H4} and \ref{A:H5}, and $g\left(t, Y_{t}^{\prime}, Z_{t}^{\prime}\right) \leq g^{\prime}\left(t, Y_{t}^{\prime}, Z_{t}^{\prime}\right), \ \ t\in[0,\tau];$\vspace{0.1cm}

(ii) $g^{\prime}$ satisfies \ref{A:H4} and \ref{A:H5}, and $g\left(t, Y_{t}, Z_{t}\right) \leq g^{\prime}\left(t, Y_{t}, Z_{t}\right), \ \ t\in[0,\tau],$\vspace{0.2cm}

\noindent then for each $t \in[0,\tau]$, we have $Y_{t} \leq Y_{t}^{\prime}.$
\end{thm}

\begin{proof}
We only prove the case (i), the case (ii) can be proved similarly. For each $n\geq1$, denote the following $(\F_t)$-stopping time:
$$\tau_{n}:=\inf \left\{t\geq0: \int_{0}^{t}e^{2\beta\int_{0}^{s}\mu_r{\rm d}r}\left(|\nu_{s}|^{2}+|Z_s|^2+|Z_s^{\prime}|^2\right){\rm d}s \geq n\right\} \wedge \tau,$$
with convention that $\inf \emptyset=+\infty$.
Let
$$U_\cdot:=Y_\cdot-Y_\cdot', \ \ V_\cdot:=Z_\cdot-Z_\cdot', \ \ \zeta:=\xi-\xi'.$$
Then $U_\cdot^+\in S_\tau^p(a_\cdot;\R^k)$ and
$$
U_{t}=\zeta+\int_{t}^{\tau} \left(g(r, Y_{r}, Z_{r})-g^{\prime}(r, Y_{r}^{\prime}, Z_{r}^{\prime})\right){\rm d}r-\int_{t}^{\tau} V_{r} {\rm d}B_{r}, \ \ t\in[0,\tau].
$$
The use of It\^{o}-Tanaka's formula to $U_{t}e^{\beta\int_{0}^{t}\mu_s{\rm d}s}$ in $[t\wedge\tau_n,\tau_n]$ yields that for each $n\geq1$ and $t\geq 0$,
\begin{align}\label{66.1}
\begin{split}
U^+_{t\wedge\tau_n}e^{\beta\int_{0}^{t\wedge\tau_n}\mu_s{\rm d}s}
\leq& U^+_{\tau_n}e^{\beta\int_{0}^{\tau_n}\mu_s{\rm d}s}-\int_{t\wedge\tau_n}^{\tau_n}e^{\beta\int_{0}^{s}\mu_r{\rm d}r}{\bf 1}_{U_s>0}V_s{\rm d}B_s\\
&+\int_{t\wedge\tau_n}^{\tau_n}e^{\beta\int_{0}^{s}\mu_r{\rm d}r}\left[{\bf 1}_{U_s>0}\left(g(s, Y_{s}, Z_{s})-g'(s, Y'_{s}, Z'_{s})\right)-\beta\mu_sU^+_s\right]{\rm d}s.
\end{split}
\end{align}
Since $g(s, Y'_{s}, Z'_{s})-g'(s, Y'_{s}, Z'_{s})$ is non-positive, combining assumptions \ref{A:H4}-\ref{A:H5} and the fact of $\beta\geq1$ we deduce that
\begin{align}\label{5.50}
\begin{split}
&{\bf 1}_{U_s>0}\left(g(s, Y_{s}, Z_{s})-g'(s, Y'_{s}, Z'_{s})\right)-\beta\mu_sU^+_s\\
&\ \  \ = {\bf 1}_{U_s>0}\left(g(s, Y_{s}, Z_{s})-g(s, Y'_{s}, Z'_{s})+g(s, Y'_{s}, Z'_{s})-g'(s, Y'_{s}, Z'_{s})\right)-\beta \mu_sU^+_s\\
&\ \  \ \leq {\bf 1}_{U_s>0}\left(\mu_s|U_s|+\nu_s|V_s|\right)-\beta \mu_sU^+_s\leq{\bf 1}_{U_s>0}\nu_s|V_s|, \ \ s\in[0,\tau].
\end{split}
\end{align}
Let
$$b_s:=\frac{\nu_sV_s^*}{|V_s|}{\bf 1}_{V_s\neq0}, \ \ s\in[0,\tau],$$
where $V_\cdot^*$ denotes the transpose of $V_\cdot$. Then, in light of \eqref{5.50}, \eqref{66.1} can be rewritten as
\begin{align}\label{5.51}
\begin{split}
U^+_{t\wedge\tau_n}e^{\beta\int_{0}^{t\wedge\tau_n}\mu_s{\rm d}s}
&\leq U^+_{\tau_n}e^{\beta\int_{0}^{\tau_n}\mu_s{\rm d}s}-\int_{t\wedge\tau_n}^{\tau_n}e^{\beta\int_{0}^{s}\mu_r{\rm d}r}{\bf 1}_{U_s>0}V_s{\rm d}B_s+\int_{t\wedge\tau_n}^{\tau_n}e^{\beta\int_{0}^{s}\mu_r{\rm d}r}{\bf 1}_{U_s>0}V_sb_s{\rm d}s\\
&= U^+_{\tau_n}e^{\beta\int_{0}^{\tau_n}\mu_s{\rm d}s}-\int_{t\wedge\tau_n}^{\tau_n}e^{\beta\int_{0}^{s}\mu_r{\rm d}r}{\bf 1}_{U_s>0}V_s\left(-b_s{\rm d}s+{\rm d}B_s\right), \ \ t\geq0.
\end{split}
\end{align}
For each $n\geq 1$, by the definition of $\tau_n$ we know that the following Novikov's condition is fulfilled:
$$
\E\left[{\rm exp}\left(\frac{1}{2}\int_0^\tau{\bf 1}_{s\leq\tau_n}|b_s|^2{\rm d}s\right)\right]<+\infty.
$$
Then, let $\mathbb{Q}_n$ be the probability measure on $(\Omega,\F_\tau)$ which is equivalent to $\mathbb{P}$ and defined by
\begin{align*}
\frac{{\rm d}\mathbb{Q}_n}{{\rm d}\mathbb{P}}:={\rm exp}\left\{\int_0^\tau{\bf 1}_{s\leq\tau_n}b_s^*{\rm d}B_s-\frac{1}{2}\int_0^\tau{\bf 1}_{s\leq\tau_n}|b_s|^2{\rm d}s\right\}.
\end{align*}
According to Girsanov's theorem, by taking the conditional mathematical expectation with respect to $\F_t$ under $\mathbb{Q}_n$ on both sides of \eqref{5.51} we can deduce that for each $n\geq1$,
\begin{align}\label{5.45}
\begin{split}
U^+_{t\wedge\tau_n}e^{\beta\int_{0}^{t\wedge\tau_n}\mu_s{\rm d}s}\leq \E_{\mathbb{Q}_n}\left[U^+_{\tau_n}e^{\beta\int_{0}^{\tau_n}\mu_s{\rm d}s}\bigg|\F_t\right]=\frac{\E\left[U^+_{\tau_n}e^{\beta\int_{0}^{\tau_n}\mu_s{\rm d}s}\frac{{\rm d}\mathbb{Q}_n}{{\rm d}\mathbb{P}}\bigg|\F_t\right]}{\E\left[\frac{{\rm d}\mathbb{Q}_n}{{\rm d}\mathbb{P}}\bigg|\F_t\right]}, \ \ t\geq 0.
\end{split}
\end{align}
Then we have
\begin{align}\label{66.6}
\begin{split}
U^+_{t\wedge\tau_n}\leq\E\left[\chi_t^n\phi_t^n\bigg|\F_t\right], \ \ t\geq 0,
\end{split}
\end{align}
where
$$
\chi_t^n:=e^{\int_{t\wedge\tau_n}^{\tau_n}\left(\beta\mu_s+\frac{\rho}{2[(p-1)\wedge1]}|b_s|^2\right){\rm d}s}U^+_{\tau_n}
$$
and
$$ \phi_t^n:=e^{-\int_{t\wedge\tau_n}^{\tau_n}\left(\frac{1}{2}+\frac{\rho}{2[(p-1)\wedge1]}\right)|b_s|^2{\rm d}s+\int_{t\wedge\tau_n}^{\tau_n}b_s^*{\rm d}B_s}.\vspace{0.1cm}
$$
Next, we prove that $(\chi_t^n\phi_t^n)^\infty_{n=1}$ is uniformly integrable for each $t\geq 0$. Indeed, by virtue of $p>1$ and $\rho>1$, we will show that there exists a constant
$$q:=\frac{PQ}{P+Q}=\frac{p\left({[(p-1)\wedge1]}+\rho\right)}
{(p+1){[(p-1)\wedge1]}+\rho}>1,$$
where $P:=p$ and $Q:=1+\frac{\rho}{[(p-1)\wedge1]}$,
such that
\begin{align}\label{66.7}
\begin{split}
\sup_{n\geq1}\E\left[(\chi_t^n\phi_t^n)^q\right]<+\infty, \ \ t\geq 0.
\end{split}
\end{align}
Since $\E[(\phi_t^n)^{Q}]=1$, by H\"{o}lder's inequality we deduce that for each $n\geq1$ and $t\geq 0$,
\begin{align*}
\begin{split}
\E\left[(\chi_t^n\phi_t^n)^{q}\right]\leq \left(\E\left[(\chi_t^n)^{P}\right]\right)^{\frac{Q}{P+Q}}
\left(\E\left[(\phi_t^n)^{Q}\right]\right)^{\frac{P}{P+Q}}
=\left(\E\left[\left(e^{\int_{t\wedge\tau_n}^{\tau_n}a_s{\rm d}s}U^+_{\tau_n}\right)^p\right]\right)^{\frac{Q}{P+Q}}.
\end{split}
\end{align*}
Note that $U_\cdot^+\in S_\tau^p(a_\cdot;\R^k)$. By taking supremum with respect to $n$ on both sides of the last inequality we deduce that
\begin{align}\label{5.47}
\begin{split}
\sup_{n\geq1}\E\left[(\chi_t^n\phi_t^n)^q\right]
\leq\left(\E\left[\sup_{r\in[0,\tau]}\left(e^{\int_{0}^{r}a_s{\rm d}s}U^+_{r}\right)^p\right]\right)^{\frac{Q}{P+Q}}<+\infty, \ \ t\geq 0,
\end{split}
\end{align}
which means that \eqref{66.7} holds with $q>1$ and then $(\chi_t^n\phi_t^n)^\infty_{n=1}$ is uniformly integrable for each $t\geq 0$. Finally, in light of $\lim\limits_{n\rightarrow\infty}U^+_{\tau_n}=U^+_{\tau}=0$, by letting $n\rightarrow \infty$ in \eqref{66.6} we know that  $\lim\limits_{n\rightarrow\infty}U^+_{t\wedge\tau_n}=0$ for each $t\geq 0$, that is to say, $Y_t\leq Y'_t, \ t\in[0,\tau]$. The proof is then complete.
\end{proof}

From \cref{thm:com1}, the following corollary is immediate.

\begin{cor}\label{cor:com1}
Assume that $p>1$, $\xi$ and $\xi'$ are two terminal values, $g$ and $g^{\prime}$ are two generators such that ${\rm d}\mathbb{P}\times{\rm d} t-a.e.$, for each $(y,z)\in\R^k\times\R^{k\times d}$, $g(t,y,z)\leq g'(t,y,z)$, and $(Y_\cdot, Z_\cdot)$ and $\left(Y_\cdot^{\prime}, Z_\cdot^{\prime}\right)$ are, respectively, a weighted $L^p$ solution of BSDE $(\xi,\tau,g)$ and BSDE $\left(\xi^{\prime}, \tau, g^{\prime}\right)$ in $ H_\tau^p(a_\cdot;\R^k\times\R^{k\times d})$. If $\xi \leq \xi^{\prime}$, and either $g$ or $g'$ satisfies \ref{A:H4} and \ref{A:H5}, then for each $t \in[0,\tau]$, we have $Y_{t} \leq Y_{t}^{\prime}$.
\end{cor}

Finally, we give the following comparison theorem on the weighted $L^1$ solutions of one-dimensional BSDEs.
\begin{thm}\label{thm:com2}
Assume that $\int_0^\tau \nu_t^2{\rm d}t\leq M$, $\xi,~\xi^{\prime} \in L_\tau^1(\beta\mu_\cdot;\R^k)$, $g$ and  $g^{\prime}$ are two generators, $(Y_\cdot, Z_\cdot)$ and $\left(Y_\cdot^{\prime}, Z_\cdot^{\prime}\right)$ are, respectively, a weighted $L^1$ solutions of BSDE $(\xi,\tau,g)$ and BSDE $\left(\xi^{\prime}, \tau, g^{\prime}\right)$ in the space of $\bigcap_{\theta\in(0,1)}[S_\tau^\theta(\beta\mu_\cdot;\R^k)$
$\times M_\tau^\theta(\beta\mu_\cdot;\R^{k\times d})]$. If $\xi \leq \xi^{\prime}$ and either of the following two conditions is satisfied:\vspace{0.2cm}

(i) $g$ satisfies \ref{A:H4}-\ref{A:H6} and $g\left(t, Y_{t}^{\prime}, Z_{t}^{\prime}\right) \leq g^{\prime}\left(t, Y_{t}^{\prime}, Z_{t}^{\prime}\right), \ \ t \in[0,\tau];$\vspace{0.1cm}

(ii) $g^{\prime}$ satisfies \ref{A:H4}-\ref{A:H6} and $g\left(t, Y_{t}, Z_{t}\right) \leq g^{\prime}\left(t, Y_{t}, Z_{t}\right), \ \ t \in[0,\tau],$\vspace{0.2cm}

\noindent then for each $t \in[0,\tau]$, we have $Y_{t} \leq Y_{t}^{\prime}.$
\end{thm}

\begin{proof}
According to \cref{thm:com1}, it suffices to prove that $(Y_\cdot-Y_\cdot^{\prime})^+$ belongs to $S_\tau^p(\beta\mu_\cdot;\R^k)$ for some $p>1$ under the assumptions of \cref{thm:com2}. We just give the proof of the case (i). The same argument is applicable for the other case.

Let us fix integer $n>1$ and define the following $(\F_t)$-stopping time:
$$\tau_n:=\inf\left\{t\geq0:\int_0^te^{2\beta\int_0^s\mu_r{\rm d}r}\left(|Z_s|^2+|Z_s'|^2\right){\rm d}s\geq n\right\}\wedge \tau,$$
with convention that $\inf \emptyset=+\infty$.
Then we have
$$
U_{t}=\zeta+\int_{t}^{\tau} \left(g(r, Y_{r}, Z_{r})-g^{\prime}(r, Y_{r}^{\prime}, Z_{r}^{\prime})\right){\rm d}r-\int_{t}^{\tau} V_{r} {\rm d}B_{r}, \ \ t\in[0,\tau],
$$
where
$$U_\cdot:=Y_\cdot-Y_\cdot', \ \ V_\cdot:=Z_\cdot-Z_\cdot', \ \ \zeta:=\xi-\xi'.$$
Applying It\^{o}-Tanaka's formula to $U_{t}e^{\beta\int_{0}^{t}\mu_s{\rm d}s}$ in $[t\wedge\tau_n,\tau_n]$ leads to, for each $n\geq1$ and $t\geq 0$,
\begin{align}\label{2266.1}
\begin{split}
U^+_{t\wedge\tau_n}e^{\beta\int_{0}^{t\wedge\tau_n}\mu_s{\rm d}s}
\leq& U^+_{\tau_n}e^{\beta\int_{0}^{\tau_n}\mu_s{\rm d}s}-\int_{t\wedge\tau_n}^{\tau_n}e^{\beta\int_{0}^{s}\mu_r{\rm d}r}{\bf 1}_{U_s>0}V_s{\rm d}B_s\\
&+\int_{t\wedge\tau_n}^{\tau_n}e^{\beta\int_{0}^{s}\mu_r{\rm d}r}\left[{\bf 1}_{U_s>0}\left(g(s, Y_{s}, Z_{s})-g'(s, Y'_{s}, Z'_{s})\right)-\beta\mu_sU^+_s\right]{\rm d}s.
\end{split}
\end{align}
In light of the fact that $g(s, Y'_{s}, Z'_{s})-g'(s, Y'_{s}, Z'_{s})$ is non-positive, it follows from assumptions \ref{A:H4} and \ref{A:H6} and the fact of $\beta\geq 1$ that
\begin{align}\label{4.46}
\begin{split}
{\bf 1}_{U_s>0}\left(g(s, Y_{s}, Z_{s})-g'(s, Y'_{s}, Z'_{s})\right)-\beta\mu_sU^+_s
\leq&2\gamma_s\left(g_s^1+g_s^2+|Y'_{s}|+|Z_{s}|+|Z'_{s}|\right)^l, \ \ s\in[0,\tau].
\end{split}
\end{align}
Combining \eqref{2266.1} and \eqref{4.46} we have for each $n\geq1$ and $t\geq 0$,
\begin{align*}
\begin{split}
U^+_{t\wedge\tau_n}e^{\beta\int_{0}^{t\wedge\tau_n}\mu_s{\rm d}s}
\leq& U^+_{\tau_n}e^{\beta\int_{0}^{\tau_n}\mu_s{\rm d}s}-\int_{t\wedge\tau_n}^{\tau_n}e^{\beta\int_{0}^{s}\mu_r{\rm d}r}{\bf 1}_{U_s>0}V_s{\rm d}B_s\\
&+\int_{t\wedge\tau_n}^{\tau_n}e^{\beta\int_{0}^{s}\mu_r{\rm d}r}2\gamma_s\left(g_s^1+g_s^2+|Y'_{s}|+|Z_{s}|+|Z'_{s}|\right)^l{\rm d}s.
\end{split}
\end{align*}
Thus, in light of assumptions of \cref{thm:com2}, similar to \eqref{5.29}-\eqref{zs}, we can deduce that for any $\theta\in(l,1)$, $U^+_\cdot$ belongs to $S_\tau^p(\beta\mu_\cdot;\R^k)$ with $p:=\frac{\theta}{l}>1$. The proof is then complete.
\end{proof}

\begin{rmk}\label{rmk5.3}
In order to prove \cref{thm:com1}, we utilize an important assertion that the bounded family of random variables in $L^p~(p>1)$ is uniformly integrable, and use H\"{o}lder's inequality skillfully, see \eqref{5.45}-\eqref{5.47} for details. We emphasize that \cref{thm:com1} improves Theorem 3.2 of  \cite{Li2023}, Proposition 5 of \cite{Briand2006}, Theorem 2.1 of \cite{SFan2016SPA} for the case of $\rho(x)=\phi(x)=x$, and that \cref{thm:com2} improves Proposition 1 of \cite{FanLiuSPL} for the case of $\alpha=1$ and Theorem 2.4 of \cite{SFan2016SPA} for the case of $\rho(x)=\phi(x)=x$. \cref{thm:com1,thm:com2} pave the way for the further study of the $L^p~(p\geq1)$ solutions for one-dimensional BSDEs with stochastic coefficients.
\end{rmk}

\section*{Acknowledgments}
The authors are grateful to anonymous referees for their patient reading and constructive suggestions, which lead to many improvements of the original manuscript. \vspace{0.3cm}

\setlength{\bibsep}{2pt}
\bibliographystyle{model5-names}

\begin{thebibliography}{99}
\bibliographystyle{plainnat}

\bibitem[Bahlali et al.(2004)]{Bahlali2004}
Bahlali K., Elouaflin A. and N'zi M., 2004. Backward stochastic differential equations with stochastic monotone coefficients. J. Appl. Math. Stoch. Anal.,  4, 317-335.

\bibitem[Bahlali et al.(2015)]{Bahlali2015}
Bahlali K., Essaky E. and Hassani M., 2015. Existence and uniqueness of multidimensional BSDEs and of systems of degenerate PDEs with superlinear growth generator. SIAM J. Math. Anal., 47(6), 4251-4288.


\bibitem[Bender and Kohlmann(2000)]{BenderKohlmann2000}
Bender C. and Kohlmann M., 2000. BSDEs with stochastic Lipschitz condition. http://hdl.handle.net/10419/85163.

\bibitem[Briand and Confortola(2008)]{BriandandConfortola2008}
Briand Ph. and Confortola F., 2008. BSDEs with stochastic Lipschitz condition and quadratic PDEs in Hilbert spaces. Stochastic Process. Appl., 118(5), 818-838.


\bibitem[Briand et al.(2003)]{Briand2003SPA}
Briand P., Delyon B., Hu Y., Pardoux E. and Stoica L., 2003. $L^p$ solutions of backward stochastic differential equations. Stochastic Process. Appl., 108(1), 109-129.

\bibitem[Briand and Hu(2006)]{Briand2006}
Briand P. and Hu Y., 2006. BSDEs with quadratic growth and unbounded terminal value. Probab. Theory Related Fields., 136(4), 604-618.





\bibitem[Chen and Wang(2000)]{ZChenBWang2000JAMS}
Chen Z. and Wang B., 2000. Infinite time interval BSDEs and the convergence of g-martingales. J. Austral. Math. Soc. Ser. A., 69(2), 187-211.



\bibitem[Delbaen and Tang(2010)]{DelbaenTang2010}
Delbaen F. and Tang S., 2010. Harmonic analysis of stochastic equations and backward stochastic differential equations. Probab. Theory Related Fields., 146(1-2), 291-336.



\bibitem[El Karoui and Huang(1997)]{KarouiHuang1997}
El Karoui N. and Huang S., 1997. A general result of existence and uniqueness of backward stochastic differential equations. Backward Stochastic Differential Equations (Paris, 1995-1996), 27-36, Pitman Research Notes
in Mathematics Series, 364, Longman, Harlow, London, UK.


\bibitem[El Karoui et al.(1997)]{KarouiEl1997}
El Karoui N., Peng S. and Quenez M., 1997. Backward stochastic differential equations in finance. Math. Finance., 7(1), 1-71.


\bibitem[Fan(2015)]{SFan2015JMAA}
  Fan S., 2015. $L^p$ solutions of multidimensional BSDEs with weak monotonicity and general growth generators. J. Math. Anal. Appl., 432(1), 156-178.

\bibitem[Fan(2016)]{SFan2016SPA}
  Fan S., 2016. Bounded solutions, $L^p~(p>1)$ solutions and $L^1$ solutions for one dimensional BSDEs under general assumptions.  Stochastic Process. Appl., 126(5), 1511-1552.


%
\bibitem[Fan and Liu(2010)]{FanLiuSPL}
Fan S. and Liu D., 2010. A class of BSDEs with integrable parameters.  Statist. Probab. Lett., 80, 2024-2031.

\bibitem[Hu et al.(2025)]{HuLiWen2025}
Hu Y., Li F. and Wen J., 2025. Multi-dimensional anticipated backward stochastic differential equations with quadratic growth. arXiv:2503.20255.

\bibitem[Ji et al.(2019)]{JiRonglinShiXueJun2019}
Ji R., Shi X., Wang Shi. and Zhou J., 2019. Dynamic risk measures for processes via backward stochastic differential equations. Insurance Math. Econom., 86, 43-50.



\bibitem[Li et al.(2021)]{LiT2019}
Li T., Xu Z. and Fan S., 2021. General time interval multidimensional BSDEs with generators satisfying a weak stochastic-monotonicity condition. Probab. Uncertain. Quant. Risk., 6(4), 301-318.

\bibitem[Li and Fan(2023)]{LiFan2023CSTM}
Li X. and Fan S., 2023. $L^p$ solutions of general time interval BSDEs with generators satisfying a  $p$-order weak stochastic-monotonicity condition. Comm. Statist. Theory Methods., 52(16), 5650-5676.

\bibitem[Li et al.(2023)]{Li2023}
Li X., Lai Y. and Fan S., 2023. BSDEs with stochastic Lipschitz condition: A general result. Probab. Uncertain. Quant. Risk., 8(2), 267-280.

\bibitem[Li et al.(2025)]{Li2024}
Li X., Zhang Y. and Fan S., 2025. Weighted solutions of random time horizon BSDEs with stochastic monotonicity and general growth generators and related PDEs. Stochastic Process. Appl., 190, 104758.


\bibitem[Liu et al.(2020)]{Liu2020}
Liu Y., Li D. and Fan S., 2020. $L^p~(p>1)$ solutions of BSDEs with generators satisfying some non-uniform conditions in $t$ and $\omega$. Chin. Ann. Math. Ser. B., 41(3), 479-494.


\bibitem[O et al.(2020)]{O2020}
O H., Kim M.-C. and Pak C.-K., 2020. A framework of BSDEs with stochastic Lipschitz coefficients. ESAIM Probab. Stat., 24,739-769.


\bibitem[{\O}ksendal(2005)]{Oksendal2005}
{\O}ksendal, B., 2005. Stochastic Differential Equations. Springer.

\bibitem[Owo(2017)]{Owo(2017)}
Owo J., 2017. $L^p$-solutions of backward doubly stochastic differential equations with stochastic Lipschitz condition and $p \in (1,2)$. ESAIM Probab. Stat., 21, 168-182.


\bibitem[Pardoux(1999)]{Pardoux1999}
Pardoux E., 1999. BSDEs, weak convergence and homogenization of semilinear PDEs. Nonlinear Analysis, Differential Equations and Control, eds. F. Clarke and R. Stern, pp. 503-549.


\bibitem[Pardoux and Peng(1990)]{PardouxPeng1990SCL}
  Pardoux E. and Peng S., 1990. Adapted solution of a backward stochastic differential equation. Syst. Control Lett., 14(1), 55-61.

\bibitem[Pardoux and R\u{a}\c{s}canu(2014)]{PardouxandRascanu(2014)}
Pardoux E. and R\u{a}\c{s}canu A., 2014. Stochastic differential equations, Backward SDEs, Partial differential equations. Springer, Cham.

%



\bibitem[Tian(2023)]{Tian2023SIAM}
Tian D., 2023. Pricing principle via Tsallis relative entropy in incomplete markets. SIAM J. Financial Math., 14(1), 250-278.


\bibitem[Wang et al.(2007)]{WangRanChen2007}
Wang J., Ran Q. and Chen Q., 2007. $L^p$ solutions of BSDEs with stochastic Lipschitz condition. J. Appl. Math. Stoch. Anal. Art.,  ID 78196, 14pp.




\bibitem[Xiao et al.(2015)]{Xiao2015}
Xiao L., Fan S. and Xu N., 2015. $L^p~(p\geq1)$ solutions of multidimensional BSDEs with monotone generators in general time intervals. Stoch. Dyn., 15(1), 1-34.



\bibitem[Yong(2006)]{Yong2006}
Yong J., 2006. Completeness of security markets and solvability of linear backward stochastic differential equations. J. Math. Anal. Appl., 319(1), 333-356.

%
%

\end{thebibliography}
\biboptions{authoryear}

\end{document}